\documentclass[11pt]{amsart}
\usepackage{amsmath,amssymb,amsthm}
\usepackage[dvips]{graphicx}
\usepackage{color}              
\usepackage{epsfig}
\usepackage{enumerate}

\usepackage{tikz}
\usetikzlibrary{arrows,decorations.pathmorphing,backgrounds,positioning,fit, through}

  \usepackage[margin=1.3in]{geometry}
\newtheorem{theorem}{Theorem}[section]
\newtheorem{proposition}[theorem]{Proposition}
\newtheorem{corollary}[theorem]{Corollary}
\newtheorem{lemma}[theorem]{Lemma}

\theoremstyle{definition}
\newtheorem{definition}[theorem]{Definition}

\theoremstyle{remark}
\newtheorem{remark}[theorem]{Remark}

\newcommand{\thmref}[1]{Theorem~\ref{#1}}
\newcommand{\propref}[1]{Proposition~\ref{#1}}
\newcommand{\secref}[1]{\S\ref{#1}}
\newcommand{\lemref}[1]{Lemma~\ref{#1}}
\newcommand{\corref}[1]{Corollary~\ref{#1}}
\newcommand{\figref}[1]{Fig.~\ref{#1}}

\newcommand{\remref}[1]{Remark~\ref{#1}}
\newcommand{\eqnref}[1]{Equation~\eqref{#1}}
\newcommand{\defref}[1]{Definition~\eqref{#1}}

\newcommand{\calB}{{\mathcal B}}
\newcommand{\calC}{{\mathcal C}}

\newcommand{\calE}{{\mathcal E}}

\newcommand{\calN}{{\mathcal N}}

\newcommand{\calP}{{\mathcal P}}

\newcommand{\calS}{{\mathcal S}}
\newcommand{\calT}{{\mathcal T}}
\newcommand{\T}{{\mathcal T}}
\newcommand{\calU}{{\mathcal U}}
\newcommand{\calV}{{\mathcal V}}
\newcommand{\calW}{{\mathcal W}}
\newcommand{\calX}{{\mathcal X}}
\newcommand{\calY}{{\mathcal Y}}

\renewcommand{\P}{{\mathbb P}}
\newcommand{\N}{{\mathbb N}}
\newcommand{\C}{{\mathbb C}}

\renewcommand{\H}{{\mathbb H}}
\newcommand{\R}{{\mathbb R}}

\renewcommand{\O}{{\mathbb O}}
\newcommand{\from}{\colon \thinspace}
\newcommand{\ST}{{\, \, \Big| \, \,}}
\newcommand{\st}{{\, \big| \,}}
\renewcommand{\d}{{\rm d}}

\newcommand{\emul}{\stackrel{{}_\ast}{\asymp}}
\newcommand{\gmul}{\stackrel{{}_\ast}{\succ}}
\newcommand{\lmul}{\stackrel{{}_\ast}{\prec}}
\newcommand{\eadd}{\stackrel{{}_+}{\asymp}}
\newcommand{\gadd}{\stackrel{{}_+}{\succ}}
\newcommand{\ladd}{\stackrel{{}_+}{\prec}}

\newcommand{\ep}{\epsilon}

\newcommand{\Teich}{Teich\-m\"u\-ller\ }
\newcommand{\Ham}{Ham\-en\-st\"adt\ }

\newcommand{\param}{{\mathchoice{\mkern1mu\mbox{\raise2.2pt\hbox{$
\centerdot$}}
\mkern1mu}{\mkern1mu\mbox{\raise2.2pt\hbox{$\centerdot$}}\mkern1mu}{
\mkern1.5mu\centerdot\mkern1.5mu}{\mkern1.5mu\centerdot\mkern1.5mu}}}

\DeclareMathOperator{\Log}{Log}

\DeclareMathOperator{\Mod}{Mod}
\DeclareMathOperator{\QI}{QI}
\DeclareMathOperator{\Isom}{Isom}
\DeclareMathOperator{\Ext}{Ext}

\DeclareMathOperator{\diam}{diam}
\DeclareMathOperator{\I}{i}

\DeclareMathOperator{\Cone}{Cone}

\newcommand{\genus}{{\sf g}}

\newcommand{\puncture}{{\sf p}}
\newcommand{\component}{{\sf c}}

\newcommand{\Bers}{{\mathsf B}}

\newcommand{\error}{{\mathsf D}}
\newcommand{\distance}{{\mathsf d}}
\newcommand{\dpinch}{{\mathsf d}_{\rm pinch}} 
\newcommand{\cb}{{\mathsf M}}

\newcommand{\thresh}{{\mathsf T}}
\newcommand{\mul}{{\mathsf K}}
\newcommand{\add}{{\mathsf C}}

\newcommand{\bp}{ {\mathbf{p}} }
\newcommand{\bx}{ {\mathbf{x}} }
\newcommand{\by}{ {\mathbf{y}} }
\newcommand{\bz}{ {\mathbf{z}} }
\newcommand{\bw}{ {\mathbf{w}} }
\newcommand{\bu}{ {\mathbf{u}} }
\newcommand{\Tt}{{\T_{\rm thick}}}
\newcommand{\fs}{f^\star}
\newcommand{\fsC}{{f^\star_\calC}}
\newcommand{\fsP}{{f^\star_\calP}}

 \newcommand{\Length}{{\mathsf L}}
 \newcommand{\Distance}{{\mathsf D}}
 \newcommand{\Dt}{{\mathsf D}_{\rm thick}}

\newcommand{\olim}{ {\text{$\omega$-$\lim$ }} }
\newcommand{\LR}{{\calT_{LR}}}
\newcommand{\balpha}{{\boldsymbol \alpha}}
\newcommand{\bbeta}{{\boldsymbol \beta}}

\newcommand{\bgamma}{{\boldsymbol \gamma}}

\newcommand{\f}{{\overline f}}
\newcommand{\bg}{{\overline g}}
\renewcommand{\MR}{{\calT_{MR}}}
\newcommand{\MC}{{\mathcal T_{\mathcal MC}}}

\begin{document}

\title{Rigidity of Teichm\"uller Space}

  \author[A.~Eskin]{Alex Eskin}
  \address{Dept. of Mathematics\\ 
  University of Chicago\\
  Chicago, Illinois, 60637}
   \email{\tt eskin@math.uchicago.edu}

  \author[H.~Masur]{Howard Masur}
  \address{Dept. of Mathematics\\ 
  University of Chicago\\
  Chicago, Illinois, 60637}
   \email{\tt masur@math.uchicago.edu}
  
  \author[K.~Rafi]{Kasra Rafi}
  \address{Dept. of Mathematics\\ 
  University of Toronto\\
  Toronto, Ontario, Canada M5S 2E4}
  \email{\tt rafi@math.toronto.edu}
  

\begin{abstract}
We prove that the every quasi-isometry of \Teich space equipped with the \Teich
metric is a bounded distance from an isometry of \Teich space. That is, 
\Teich space is quasi-isometrically rigid. 
\end{abstract}
\maketitle

\section{Introduction and Statement of the theorem}
In this paper we continue our study of the coarse geometry of \Teich space 
begun in \cite{rafi:CD}. Our goal, as part of Gromov's broad program to understand 
spaces and groups  by their coarse or quasi-isometric  geometry, is to carry this 
out in the context of \Teich space equipped with the \Teich metric. To state the 
main theorem, let $S$ be a connected surface of finite hyperbolic type.  
Define the complexity of $S$ to be
\[
\xi(S) =3\genus(S) + \puncture(S) -3,
\]
where $\genus(S)$ is the genus of $S$ and $\puncture(S)$ is the number 
of punctures. Let $\T(S)$ denote the  \Teich space of $S$ equipped with the 
\Teich metric $d_{\T(S)}$.

\begin{theorem} \label{Thm:Main}
Assume $\xi(S) \geq 2$. Then, for every $\mul_S, \add_S >0$ there is a constant 
$\error_S >0$ such that if 
\[f_S \from \T(S) \to \T(S)\] 
is a  $(\mul_S, \add_S)$--quasi-isometry then there is an isometry 
\[ \Psi_S \from \T(S) \to \T(S)\] 
such that, for $x \in \calT(S)$
\[
d_{\T(S)} \big( f(x), \Psi(x) \big) \leq \Distance_S. 
\]
\end{theorem}

Recall that a map $f \from \calX \to \calY$ from a metric space $(\calX, d_\calX)$
to a metric space $(\calY, d_\calY)$ is a $(K, C)$--quasi-isometry if it is $C$--coarsely 
onto and, for $x_1,x_2 \in \calX$,
\begin{equation} \label{Eq:QI}
\frac 1K d_\calX(x_1,x_2) - C \leq d_\calY\big( f(x_1), f(x_2)\big) \leq K d_\calT(x_1,x_2) + C.
\end{equation}
If \eqnref{Eq:QI} holds and the map is not assumed to be onto, then $f$ is called a {\em quasi-isometric embedding}.  One defines, for a metric space $(\calX,d_\calX)$, the group $\QI(\calX)$ as the 
equivalence classes of quasi-isometries from $\calX$ to itself, with two 
quasi-isometries being equivalent if they are a bounded distance apart. 
When the natural homomorphism $\Isom(\calX) \to \QI(\calX)$ is a isomorphism,  
we say $\calX$ is quasi-isometrically rigid. Then the Main Theorem restated is 
that $\T(S)$ is quasi-isometrically rigid.

Note also that, by Royden's Theorem \cite{royden:AT}, $\Isom\big(\T(S)\big)$ is 
essentially the mapping class group (the exceptional cases are the twice-punctured
torus and the closed surface of genus $2$ where the two groups differ by a finite
index). Hence, except for the lower complexity cases, $\QI\big(\T(S)\big)$ is 
isomorphic to the mapping class group. 
This theorem has also been proven by Brian Bowditch \cite{bowditch:RT} by a different method. 

\subsection*{History and related results}
There is a fairly long history that involves the study of the group $\QI(\calX)$ in 
different contexts. Among these, symmetric spaces are the closest to our setting. 

In the case when $\calX$ is $\R^n$, $\H^n$ or $\C\H^n$, the quasi-isometry group 
is complicated and much larger than the isometry group. 
In fact, if a self map of $\H^n$ is a bounded distance from an isometry then it induces 
a conformal map on $S^n$. But, every quasi-conformal homeomorphism from 
$S^n\to S^n$  extends to a quasi-isometry of $\H^n$. 
This, in particular, shows why the condition $\xi(S) \geq 2$ in \thmref{Thm:Main} 
is necessary. When $\xi(S) =1$, $\calT(S)$ is isometric (up to a factor of $2$) to 
the hyperbolic plane $\H$ and, as mentioned above, $\H$ is not rigid. 

Pansu \cite{pansu:MC} proved that other rank one symmetric spaces of non-compact 
type such as quaternionic hyperbolic space $\H\H^n$  and the Cayley plane $\P^2(\O)$ 
are rigid. 
In contrast, higher rank irreducible symmetric spaces are rigid. This was 
proven by Kleiner-Leeb \cite{kleiner:RS} (see also Eskin-Farb 
\cite{eskin:RH} for a different proof). In our setting, when $\xi(S) \geq 2$, $\T(S)$ is analogous
to a higher rank symmetric space (see \cite{rafi:CD} and \cite{bowditch:RT}
for discussion about flats in $\T(S)$) and the curve complex $\calC(S)$, 
which plays a prominent role in \Teich theory is analogous to the Tits boundary 
of symmetric space. 

Continuing the above analogy, the action of mapping class group on \Teich
space is analogous to the action of a non-uniform lattice on a symmetric space. 
Quasi-isometric rigidity was shown for non-uniform lattices in rank $1$ groups other  
than $SL(2,\R)$  \cite{schwartz:QRO} and for some higher rank lattices \cite{schwartz:RD}
by Schwartz and in general by Eskin \cite{eskin:QNL}. 
The quasi-isometric rigidity of the  mapping class group $\Mod(S)$ was shown by 
Behrstock-Kleiner-Minsky-Mosher \cite{minsky:MR} and by \Ham \cite{hamenstadt:QR} 
and later by Bowditch \cite{bowditch:RM}. 
More generally Bowditch in that same paper showed that if $S$ and $S'$ are closed 
surfaces  with $\xi(S)=\xi(S')\geq 4$ and $\phi$ is a  quasi-isometric embedding 
of $\Mod(S)$ in $\Mod(S')$, then $S=S'$ and $\phi$ is bounded distance from an 
isometry.  He also shows quasi-isometric rigidity for Teichm\"uller 
space with the Weil-Petersson metric \cite{bowditch:RW}.

\subsection*{Inductive step}
We prove this theorem inductively. To apply induction, we need to consider 
non-connected surfaces. Let $\Sigma$ be a possibly disconnected surface of finite 
hyperbolic type. We always assume that $\Sigma$ does not have a component that
is a sphere, an annulus, a pair of pants or a torus. We define the complexity 
of $\Sigma$ to be
\[
\xi(\Sigma) =3\genus(\Sigma) + \puncture(\Sigma) -3\component(\Sigma),
\]
where $\component(\Sigma)$ is the number of connected components of $\Sigma$. 
For point $x \in \T(\Sigma)$,  let $P_x$ be the short pants decomposition at $x$
(see \secref{Sec:Short}) and for a curve $\gamma \in P_x$ define 
(see \secref{Sec:Short} for the exact definition) 
\[
\tau_x(\gamma) \simeq \log \frac {1}{\Ext_x(\gamma)}. 
\]
For a constant $L>0$, consider the sets
\[
\T(\Sigma, L) = \Big\{ x \in \T(\Sigma) \ST \text{for every curve $\gamma$}, \quad
\tau_x(\gamma) \leq L \Big\},
\]
and
\[
\partial_L(\Sigma, L) = \Big\{ x \in \T(\Sigma, L) \ST 
\forall\, \gamma \in P_x, \quad \tau_x(\gamma) = L \Big\}. 
\]
Thinking of $L$ as a very large number, we say a quasi-isometry 
\[
f_\Sigma \from \T(\Sigma, L) \to \T(\Sigma, L)
\] 
is $\add_\Sigma$--anchored at infinity if the restriction of $f_\Sigma$ to 
$\partial_L(\Sigma)$ is nearly the identity. That is, for every $x \in \partial_L(\Sigma, L)$, 
\[
d_{\T(\Sigma)} \big( x, f_\Sigma(x) \big) \leq \add_\Sigma. 
\]
Our induction step is the following. 

\begin{theorem} \label{Thm:Induction}
Let $\Sigma$ be a surface of finite hyperbolic type. For every $\mul_\Sigma$ and 
$\add_\Sigma$, there is $\Length_\Sigma$ and $\Distance_\Sigma$ so that, for
$L \geq \Length_\Sigma$, if 
\[
f_\Sigma \from \T(\Sigma, L) \to \T(\Sigma, L) 
\]
is a $(\mul_\Sigma, \add_\Sigma)$--quasi-isometry that is $\add_\Sigma$--anchored at infinity, 
then for every $x \in \T(\Sigma)$ we have
\[
d_{\T(\Sigma)} \big( x, f_\Sigma(x) \big) \leq \Distance_\Sigma.
\]
\end{theorem}

Note that $\Distance_\Sigma$ depends on the topology of $\Sigma$ and constants 
$\mul_\Sigma$, $\add_\Sigma$ and $\Length_\Sigma$, but is independent 
of $L$.

\subsection*{Outline of the proof}
The overall strategy is to take the quasi-isometry and prove it preserves more and 
more of the structure of \Teich space.  Section~\ref{Sec:Background} is devoted to 
establishing notation and background material.  In Section~\ref{Sec:Coarse} we define 
the rank of a point as the number of short curves plus the number of complimentary 
components that are not pairs of pants. A point has maximal rank if all complimentary 
components $W$ have $\xi(W)=1$.  We show in \propref{Prop:Rank-Preserved} 
that a quasi-isometry preserves points with maximal rank.  The proof uses the ideas of 
coarse differentiation, previously developed in the context of \Teich space in \cite{rafi:CD}
which in turn is based on the work of Eskin-Fisher-Whyte \cite{eskin:CDI, eskin:CDII}. 
Coarse differentiation was also previously used by Peng to study  quasi-flats in solvable 
Lie groups in \cite{peng:CDI} and \cite{peng:CDII}.
The important property  of maximal rank is that near such a point of maximal rank, 
\Teich space is close to being isometric to  a product of copies of $\H$, with the 
supremum metric. 

In Section~\ref{Sec:Local} we prove a local version of the splitting theorem 
shown in \cite{kleiner:RS} in the context of symmetric spaces and later in 
\cite{eskin:QF} for products of hyperbolic planes. There it is proved in 
\thmref{Thm:Factor-Preserving}  that a quasi-isometric embedding from a large ball in 
$\prod \H$ to $\prod \H$ can be restricted to a smaller ball where it factors, up to a a 
fixed additive error. This local factoring is applied in Section~\ref{Sec:Local-Factors} 
to give a bijective association $\fs_x$ between factors at $x$ and at $f(x)$. 
We also prove a notion of analytic continuation, namely, we examine when local 
factors around points $x$ and $x'$ overlap, how $\fs_x$ and $\fs_{x'}$ are related. 
  
We use this to  show, Proposition~\ref{Prop:Max-Preserved},  that  maximal cusps 
are preserved by the quasi-isometry. A maximal cusp is the subset of  maximal rank 
consisting of points  where there is a maximal set of short curves all about the 
same length. There \Teich space looks like a cone in a product of horoballs inside 
the product of $\H$.  The  set of maximal cusps is disconnected;  there is 
a component associated to every pants decomposition. Thus, $f$ induces
a bijection on the set of pants decomposition. The next step is then to show  this map 
is induced by automorphism of the curve complex, and hence by Ivanov's Theorem, 
it is associated to an isometry of \Teich space. Composing $f$ by the inverse, we can assume that $f$ sends every component of the set of maximal cusps
to itself. An immediate consequence of this is that $f$ restricted to the  thick part of \Teich space
is a bounded distance away from the identity (\propref{Prop:Thick}). 
From this  fact and again applying Propositions~\ref{Prop:Local-Factors} and 
\ref{Prop:Analytic-Continuation}, we then show, \propref{Prop:Short-Curve},  
that for any point in the maximal rank set, the shortest curves are  preserved
and in fact, for any shortest curve $\alpha$ at points $x$, $\fs_x(\alpha) = \alpha$.  

This allows one to cut along the shorest curve, induce a quasi-isometry on 
\Teich space  of lower complexity and proceed by induction. Most of the discussion
above also applies for the disconnected subsurfaces. Hence, much of 
Sections~\ref{Sec:Local-Factors} and \ref{Sec:Short-Curves} is written in a way
to apply to both $\T(S)$ and $\T(\Sigma, L)$ settings. The induction step is carried out 
in Section~\ref{Sec:Induction}.
 
\subsection*{Acknowledgements} We would like to thank Brian Bowditch for
helpful comments on an earlier version of this paper. 

\section{Background} \label{Sec:Background} 
The purpose of this section is to establish notation and recall some statements
from the literature. We refer the reader to \cite{hubbard:TT} and \cite{farb:MCG} for 
basic background on \Teich theory and to \cite{minsky:PR} and \cite{rafi:HT} for
some background on the geometry of the \Teich metric. 

For much of this paper, the arguments are meant to apply to both
$\T(S)$ which is the \Teich space of a connected surface and to 
$\T(\Sigma, L)$ where $\Sigma$ is disconnected and the
space is truncated at infinity. In such situations, we use the notation $\T$ to refer to 
either case and use the full notation where the discussion is specific to one
case or the other. Similarly, $f$ denotes either a quasi-isometry $f_S$
of $\T(S)$ or a  quasi-isometry $f_\Sigma$ of $(\Sigma, L)$. A similar convention 
is also applied to other notations as we suppress symbols
$S$ or $\Sigma$ to unify the discussion in the two cases. For example,
$\mul$ and $\add$ could refer to $\mul_S$ and $\add_S$ or to 
$\mul_\Sigma$ and $\add_\Sigma$ and $\xi$ could refer to $\xi(S)$ or $\xi(\Sigma)$. 

By a \emph{curve}, we mean the free isotopy class of a non-trivial, non-peripheral 
simple closed curve in either $S$ or $\Sigma$. Also, by a subsurface
we mean a free isotopy class of a subsurface $U$, where the inclusion map induces
an injection between the fundamental groups. We always assume that $U$ is not a pair 
of pants. When we say $\gamma$ is a curve in $U$, we always assume that it is not 
peripheral in $U$ (not just in $S$). We write $\alpha \subset \partial U$ 
to indicate that the curve $\alpha$ is a boundary component of $U$. 

\subsection{Product Regions} \label{Sec:PR} 
We often examine a point $x \in \T$ from the point of view of its subsurfaces. 
For every subsurface $U$, there is a projection map
\[
\psi_U \from \T \to \T(U)
\] 
defined using the Fenchel-Nielson coordinates (see \cite{minsky:PR} for details).
When the boundary of $U$ is not short, these maps are not well behaved.
Hence these maps should be applied only when there is an upper-bound on the length
of $\partial U$ (see below). 

When $U$ is an annulus, or when $\xi(U)=1$, the space $\T(U)$ can be
identified with the hyperbolic plane $\H$;  however the \Teich metric differs 
from the usual hyperbolic metric by a factor of $2$. We always assume that $\H$ 
is equipped with this metric which has  constant curvature $-4$. We denote
the $\psi_U(x)$ simply by $x_U$. 

A \emph{decomposition} of $S$ is a set $\calU$ of pairwise disjoint subsurfaces 
of $S$ that fill $S$. Subsurfaces in $\calU$ are allowed to be annuli (but not
pairs of pants). In this context, \emph{filling} means that every curve in $S$
either intersects or is contained in some $U \in \calU$. In particular, for every
$U \in \calU$, the annulus associated to every boundary curve of $U$ is also
included in $\calU$. The same discussion holds for $\Sigma$. 

For a decomposition $\calU$ and $\ell_0>0$ define
\[
\T_\calU = \Big\{ x \in \T
   \ST \forall \alpha \subset \partial U, \quad \Ext_x(\alpha) \leq \ell_0 \Big\}.
\]
\begin{theorem}[Product Regions Theorem, \cite{minsky:PR}] 
\label{Thm:Product-Regions} 
For $\ell_0$ small enough, 
\[
\psi_\calU = \prod_{U \in \, \calU} \psi_U
\from \T_\calU \to \prod_{U \in \, \calU} \T(U)
\]
is an isometry up to a uniform additive error $\Distance_0$. Here, the product on 
the right hand side is equipped with the sup metric. That is, for $x^1, x^2$
\begin{equation} \label{Eq:PR}
d_\T(x^1, x^2) - \Distance_0 \leq \sup_{U \in \, \calU} d_{\T(U)}(x^1_U, x^2_U)
\leq d_\T(x^1, x^2) + \Distance_0.
\end{equation}
\end{theorem}

For the rest of the paper, we fix the value of $\ell_0$ that makes this statement 
hold. We also assume that two curves of length less than $\ell_0$ do not
intersect. We refer to $\T_\calU$ as \emph{the product region associated to 
$\calU$}.  

We say $\calU$ is \emph{maximal} if every $U \in \calU$ is either an annulus, 
or $\xi(U)=1$ (recall that pairs of pants are always excluded). That is, if $\calU$ 
is maximal then the associated product region is isometric, up to an additive error 
$\Distance_0$, to a subset of a product space $\prod_{i=1}^\xi \H$ equipped with 
the sup metric. For the rest of the paper, we always assume the product 
$\prod_{i=1}^\xi \H$ is equipped with the sup metric. 
For points $x, y \in \T(S)$, we say  \emph{$x$ and $y$ are in the same maximal 
product region} if there is a maximal decomposition $\calU$
where $x, y \in \T_\calU$. Note that such a $\calU$ is not unique. For example, 
let $P$ be a pants decomposition and let $x$ be a point where the length in $x$ 
of every curve in $P$ is less than $\ell_0$. Then there are many decompositions 
$\calU$ where every $U \in \calU$ is either a punctured torus or a four-times punctured 
sphere with $\partial U \subset P$ or an annulus whose core curve is in $P$. 
The point $x$ belongs to  $\T_\calU$ for every such decomposition $\calU$. 

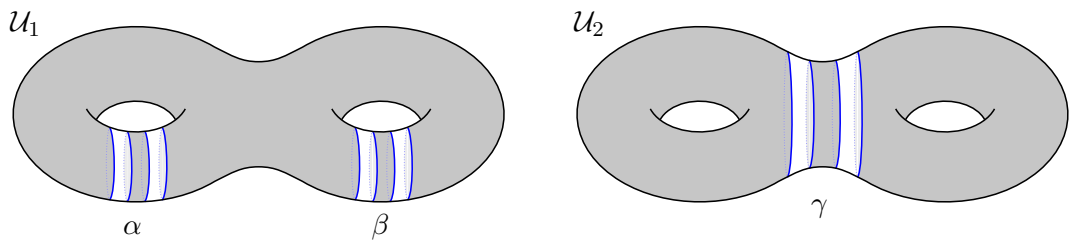
\begin{figure}[ht]
\setlength{\unitlength}{0.01\linewidth}
\begin{picture}(100, 20)
\put(1,0){\includegraphics[width=48\unitlength]{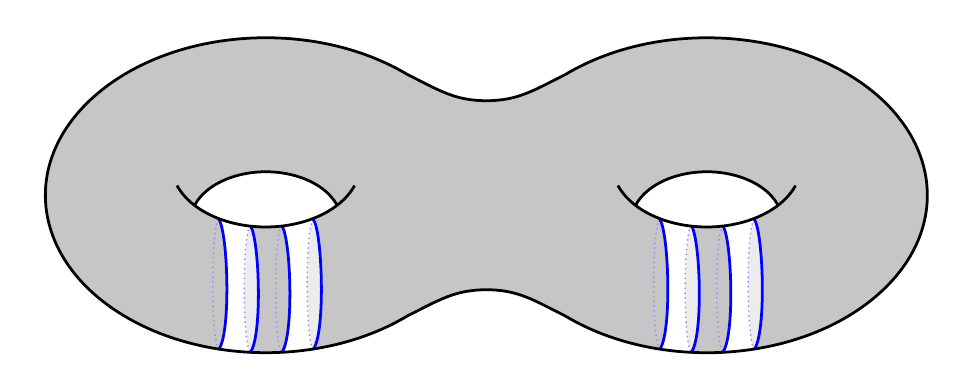}}
\put(51,0){\includegraphics[width=48\unitlength]{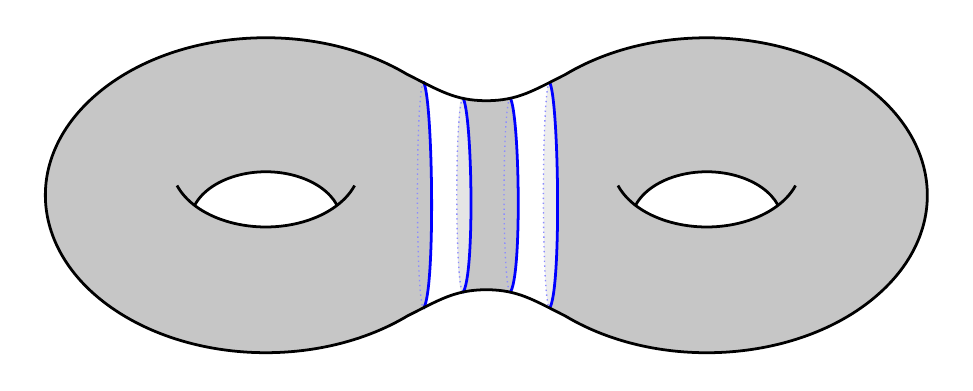}}

\put(3,17){$\calU_1$}
\put(53,17){$\calU_2$}

\put(13, -1){$\alpha$} 
\put(35, -1){$\beta$}
\put(74, 1){$\gamma$}

\end{picture} 
\caption{Two maximal decompositions $\calU_1$ and $\calU_2$ are depicted above.
Note that $\calU_1 \cap \, \calU_2 = \emptyset$. However, for the pants decomposition 
$P = \{ \alpha, \beta, \gamma \}$, any point $x \in \T$ where all curves in 
$P$ have a length less than $\ell_0$ is contained in $\T_{\calU_1} \cap \T_{\calU_2}$.} 
\label{Fig:Surface} 
\end{figure}

\subsection{Short curves on a surface} \label{Sec:Short} 
The thick part $\Tt$ of $\T$ is the set of points $x$
where, for every curve $\gamma$, $\Ext_X(\gamma) \geq \ell_0$. 
There is a constant $\Bers$ (the Bers constant) so that, for any point $z \in \Tt$, 
the set of curves that have extremal length at most $\Bers$ \emph{fills} the surface. 
That is, every curve intersects a curve of length at most $\Bers$. 
Note that every $x \in \T$ contains a curve of length at most $\Bers$. 
In fact, we choose $\Bers$ large enough so that every $x \in \T$ has a
pants decomposition of length at most $\Bers$. We call such a pants decomposition
the \emph{short pants decomposition at $x$} and denote it by $P_x$. 
We also assume that $\ell_0$ is small enough so that if $\Ext_x(\alpha) \leq \ell_0$
($\alpha$ is $\ell_0$--short) then $\alpha$ does not intersect any 
$\Bers$--short curve. Hence $P_x$ contains every $\ell_0$--short curve. 

It is often more convent to work with the logarithm of length. For $x \in T$
and $\alpha \in P_x$, define $\tau_x(\alpha)$ to be the largest number so that
if $d_\T(x, x') \leq \tau_x(\alpha)$ then $\Ext_{x'}(\alpha) \leq \ell_0$. 
From the product regions theorem, we have
\[
\left| \tau_x(\alpha) - \log \frac 1{\Ext_x(\alpha)} \right| = O(1),
\]
where the constant on the right hand side depends on the value of 
$\Distance_0$ and $\log 1/\ell_0$. We often need to \emph{pinch} a curve. 
Let $x \in \T$ and $\tau$ be given and let $\alpha\in P_x$ with $\tau_x(\alpha)=O(1)$. 
Let $x'$ be a point with $d_\T(x, x')=O(1)$ and where the length of $\alpha$ is $\ell_0$. 
Let $U = S - \alpha$ and let $x''$ be the point so that 
\[
\tau_{x''}(\alpha) = \tau, \qquad
x'_U =x''_U \qquad\text{and}\qquad
\Re (x'_\alpha) = \Re (x''_\alpha). 
\]
The last condition means $x'$ and $x''$ have no relative twisting around $\alpha$. 
We then say \emph{$x''$ is a point obtained from $x$ by pinching $\alpha$}. 
There is a constant $\dpinch$ so that 
\[
d_\T(x, x'') \leq \tau + \dpinch. 
\]

\subsection{Subsurface Projection} \label{Sec:Subsurface} 
Let $U$ be a subsurface of $S$ with $\xi(U) \geq 1$. Let $\calC(U)$ denote the curve
graph of $U$; that is, a graph where a vertex is a curve in $U$, and an edge is a pair 
of disjoint curves. When $\xi(U) =1$, $U$ does not contain disjoint curves. 
Here an edge is a pair of curves that intersect minimally; once in the punctured-torus 
case and twice in the four-times-punctured sphere case. 
In the case $U$ is an annulus with core curve $\alpha$, in place 
of the curve complex, we use the subset $H_\alpha \subset \T(U)$ of all points where
the extremal length of $\alpha$ is at most $\ell_0$. This is a horoball in $\H=\T(U)$.
Depending on context, we use the notation $\calC(U)$ or $H_\alpha$. 

There is a projection map 
\[
\pi_U \from \T(U) \to \calC(U)
\] 
that sends a point $z \in \T(U)$ to a curve $\gamma$ in $U$ with 
$\Ext_z(\gamma) \leq \Bers$. This is not unique but the image has a uniformly bounded 
diameter and hence the map is coarsely well defined. When $U$ is an annulus, 
$\pi_U(z)$ is the same as $\psi_\calU(z)$ if $\Ext_z(\alpha) \leq \ell_0$ and, otherwise, 
is the point on the boundary of $H_\alpha$ where the real value is 
twisting of $z$ around $\alpha$. (see \cite{rafi:HT} for the definition and 
discussion of twisting). 

We can also define a projection $\pi_U(\gamma)$ where $\gamma$ is any curve
that intersects $U$ non-trivially. If $\gamma\subset U$ then choose the projection to 
be $\gamma$.  If $\gamma$ is not contained in $U$ then $\gamma\cap U$ is a 
collection of arcs  with endpoints on $\partial U$.  Choose one such arc and perform a 
surgery using this arc and a sub-arc of $\partial U$ to find a point in $\calC(U)$. 
The choice of different arcs or different choices of intersecting pants curves 
determines a set of diameter $2$ in $\calC(U)$; hence the projection is coarsely 
defined. Note that this is not defined when $\gamma$ is disjoint from $U$. 
We also define a projection $\T(S) \to \calC(U)$ to be $\pi_U \circ \pi_S$,
however we still denote it by $\pi_U$. For $x,y \in \T$, we define
\[
d_U(x,y) := d_{\calC(U)}\big(\pi_U(x), \pi_U(y) \big).
\]
For curves $\alpha$ and $\beta$, $d_U(\alpha, \beta)$ is similarly defined. 

In fact, the subsurface projections can be used to estimate the distance
between two points in $\T$ (\cite{rafi:CM}). There is threshold $\thresh$ so that  
\begin{equation} \label{Eq:Distance}
d_\T(x,y) \emul \sum_{W \in \calW_\thresh} d_W(x,y),
\end{equation}
where $\calW_\thresh$ is the set of subsurfaces $W$ where $d_W(x,y) \geq \thresh$. 

\begin{definition} \label{Def:CB}
We say a pair of points $x,y \in \T(S)$ are $M$--cobounded relative to
a subsurface $U \subset S$ if $\partial U$ is $\ell_0$--short in
$x$ and $y$ and if, for every surface $V \neq U$, 
$d_V(x,y)\leq M$. If $U=S$, we simply say $x$ and $y$ are
$M$--cobounded. 

Similarly, we say a pair of curves $\alpha,\beta$ are $M$-- cobounded relative to 
$U$ if for every $V\subsetneq U$, $d_V(\alpha,\beta)\leq M$ when defined.  
If $U=S$, we simply say $\alpha$ and $\beta$ are $M$--cobounded. 

A path $g$ in $\T(S)$ or in $\calC(U)$ is $M$--cobounded relative to $U$ 
if every pair of points in $g$ are $M$--cobounded relative to $U$.  Once and for 
all, we choose a constant $\cb$ so that through every point $x \in \T(S)$ and for every 
$U$ whose boundary length is at most $\ell_0$ in $x$, there is a bi-infinite path 
in $\T(S)$ passing through $x$ that is $\cb$--cobounded relative to $U$.
One can, for example take an axis of a pseudo-Anosov element in $\T(U)$
and then use the product regions theorem to elevate that to a path in $\T$
whose projections to other subsurfaces are constant. When we say 
a geodesic $g$ in $\T(S)$  is cobounded relative to $U$, we always mean that it is 
$\cb$--cobounded relative to $U$. From  \eqnref{Eq:Distance} (the distance formula) 
we have, for $x$ and $y$ along such $g$, 
\begin{equation}
\label{eq:distance}
d_\T(x,y) \emul d_U(x,y). 
\end{equation}
\end{definition}
%
%

\section{Rank is Preserved} \label{Sec:Coarse}
In this section, we recall some results from \cite{rafi:CD} and we develop 
them further to show that the set of points in \Teich space
with maximum rank is coarsely preserved (see \propref{Prop:Rank-Preserved} below). 
Since the notation $\emul$ and $\gmul$ where used in \cite{rafi:CD}, we continue 
to use them in this section. Recall from \cite{rafi:CD} that 
${\mathbf A} \emul {\mathbf B}$ means there is a constant $C$, depending
only on the topology of $S$ or $\Sigma$, so that 
$\frac{\mathbf A}C \leq {\mathbf B} \leq C {\mathbf A}$. We say ${\mathbf A}$
and ${\mathbf B}$ are \emph{comparable}. Similarly, 
${\mathbf A} \eadd {\mathbf B}$ means there is a constant $C$, depending
only on the topology of $S$ or $\Sigma$, so that 
${\mathbf A} - C \leq {\mathbf B} \leq {\mathbf A}+C$. We say ${\mathbf A}$
and ${\mathbf B}$ are \emph{the same up to an additive error}.
For the rest of the paper, we would need to be more careful with constant. 
The only constants from this section that is used later is the constant $\distance_0$ 
from \propref{Prop:Rank-Preserved}. 

\subsection*{Coarse Differentiation and Preferred Paths} 
A path $g \from[a,b] \to \T$ is called \emph{a preferred path} if, 
for every subsurface $U$, the image of $\pi_U \circ g$ is a reparametrized
quasi-geodesic in $\calC(U)$. We use preferred paths as coarse analogues of 
straight lines in $\T$. 

\begin{definition} \label{Def:size}
A \emph{box} in $\R^n$ is a product of intervals; namely
$B= \prod_{i=1}^n I_i$, where $I_i$ is an interval in $\R$. We say
a box $B$ is \emph{of size $R$} if, for every $i$, $|I_i| \gmul R$ and if the diameter 
of $B$ is less than $R$. Note that if $B$ is of size $R$ and of size $R'$, then
$R \emul R'$. The box in $\R^n$ is always assumed to be equipped with the 
usual Euclidean metric. 

For points $a,b \in B$, we often treat the geodesic segment $[a,b]$ in $B$
as an interval of times parametrized by $t$. A map $f \from B \to \T$ from a box 
of size $R$ in $\R^n$ to $\T$ is called \emph{$\ep$--efficient} if, for 
any pair of points $a, b \in B$, there is a preferred path $g \from [a,b] \to \T$ so that, 
for $t \in [a,b]$
\[
d_\T\big( f(t), g(t)\big) \leq \ep R.
\]
\end{definition}

Let $B$ be a box of size $L$ in $\R^n$ and let $\underline B$ be
a central sub box of $B$ with comparable diameter (say a half). 
For any constant $0<R\leq L/3$, let $\calB_R$ be a subdivision $\underline B$ to 
boxes of size $R$. That is, 
\begin{enumerate}
\item boxes in $\calB_R$ are of size $R$, 
\item they are contained in $\underline B$ and hence their distance
to the boundary of $B$ is comparable to $L$, 
\item they have disjoint interiors and 
\item $| \calB_R| \emul (L/R)^n$.
\end{enumerate}
The following combines Theorem~2.5 and Theorem~4.9 in \cite{rafi:CD}.
 
\begin{theorem}[Coarse Differentiation \cite{rafi:CD}]  \label{Thm:Differentiable}
For every $K,C, \ep, \theta$ and $R_0$ there is $L_0$ so that the following 
holds. For $L \geq L_0$, let $f \from B \to \T$ be a $(K, C)$--quasi-Lipschitz 
map where $B$ is a box of size $L$ in $\R^n$. Then, there is a scale $R \geq R_0$
so that the proportion of boxes $B' \in \calB_R$ where $f|_{B'}$ is
$\ep$--efficient is at least $(1-\theta)$. 
\end{theorem} 

Even though we have no control over the distribution of efficient boxes,
the following Lemma says we can still connect every two point in $\underline B$ 
with a path that does not intersect too many non-efficient boxes. 

\begin{lemma} \label{Lem:Path}
Let $L$, $R$, $\calB_R$, $\ep$ and $\theta$ be as above.  
Then, for any pair of points $a,b \in \underline B$, there is a path $\gamma$ in 
$\underline B$ connecting them so that $\gamma$ is covered by at most $O(L/R)$ 
boxes and the number of boxes in the covering that are not $\ep$--efficient is 
at most $O\big(\sqrt[n]{\theta} \frac LR\big)$. 
\end{lemma}

\begin{proof}
Let $N = \sqrt[n]{\theta} \frac LR$. First assume that the distance between $a$ an $b$ 
to the boundary of $\underline B$ is at least $NR$. Consider the geodesic segment 
$[a,b]$. Take a $(n-1)$--dimensional totally geodesic boxes $Q_a$ and $Q_b$ 
containing $a$ and $b$ respectively that are perpendicular to $[a,b]$, parallel to 
each other and have a diameter $N R$. Choose an $R$--net of points 
$p_1, \ldots, p_k$ in $Q_x$ and $q_1, \ldots, q_k$ in $Q_y$ so that $[p_i, q_i]$ 
is parallel to $[a,b]$. We have, for $1 \leq, i,j \leq k$, 
\[
d_{\R^n}\big( [p_i, q_i], [p_j, q_j] \big) > R
\qquad\text{and}\qquad
k \emul N^{n-1}.
\]
For $1 \leq i \leq k$, let $\gamma_i$ be the path that is a concatenation 
of  geodesic segments $[a,p_i]$, $[p_i, q_i]$ and $[q_i, b]$. We claim one of
these paths has the above property. 

Assume, for contradiction, that the number of non-efficient boxes along each 
$\gamma_i$ is larger than $cN$ for some large $c>0$. Then the total number of 
non-efficient boxes is at least
\[
k c N = c N^n = c \, \theta \left( \frac LR \right)^n. 
\]
But this is not possible for large enough value of $c$ (see property (4) of 
$\calB_R$ above and \thmref{Thm:Differentiable}). Hence, there is a $c=O(1)$ and  $i$ where 
$[p_i, q_i]$ intersects at most $c N$ non-efficient boxes. 

Note that, the segments $[a,p_i]$ and $[q_i,b]$ intersect at most 
$N$ boxes each. Hence the number of inefficient boxes intersecting $\gamma_i$ 
is at most $(c+2)N$. In the case $a$ or $b$ are close to the boundary, we choose 
points $a'$ and $b'$
nearby (distance $NR$) and apply the above argument to find an appropriate path
between $a'$ and $b'$ and then concatenate this path with segments $[a,a']$
and $[b,b']$. The total number of inefficient boxes along this path is at most
$(c+4)N$. This finishes the proof. 
\end{proof}

\subsection*{Efficient quasi-isometric embeddings}
In this section, we examine efficient maps that are also assumed to be
quasi-isometric embeddings. We will show that they have \emph{maximal rank};
they make small progress in any subsurface $W$ with $\xi(W) \geq 2$. 

\begin{definition} \label{Def:StandardFlat}
Let $\calU$ be decomposition of $S$. For every $U \in \calU$, let 
$g_U \from I_U \to \T(U)$ be a preferred path. Consider the box 
$B = \prod_U I_U \subset \R^m$, where $m$ is the number of elements in $\calU$. 
Consider the map 
\[
F\from B \to \T_\calU= \prod_{U \in \calU} \T(U)
\qquad\text{where}\qquad
F= \prod_{U \in \calU} g_U.
\] 
Then $F$ is a quasi-isometric embedding because each $g_U$ is a 
quasi-geodesic. We call this map a \emph{standard flat}
in $\T_\calU$.
\end{definition}

The map $\f$ below will be a modified version our map $f$ from \thmref{Thm:Main}.

\begin{theorem} \label{Thm:Tunnel}
For every $K$, $C$ and $M$, there is $\ep$ and $R_0$ so that 
the following holds. Let $\f \from B \to \T(S)$ be an $\ep$--efficient 
$(K,C)$--quasi-isometric embedding defined on a box $B\subset\R^{\xi}$ of size 
$R \geq R_0$, let $\omega_0$ be a $M$--cobounded geodesic in $\calC(W)$, 
where $W$ is a subsurface with $\xi(W)\geq 2$, and let $\pi_{\omega_0}$ be the 
closest point projection map from $\calC(W)$ to $\omega_0$. Then, for
\[
\pi=\pi_{\omega_0}\circ \pi_W\circ \f
\]
and points $a,b\in B$, we have
\[
d_W\big(\pi(a),\pi(b) \big)\leq \sqrt \ep R.
\]
\end{theorem}

\begin{proof}
Assume by way of contradiction that for all large $R$ and all small $\epsilon$ there is a  
subsurface $W$, $\xi(W) \geq 2$, an $M$--cobounded geodesic $\omega_0$ in $\calC(W)$, 
a box $B$ of size $R$ in $\R^m$, and an $\ep$--efficient map $\f \from B \to \T$ 
and a pair of points $a,b\in B$ such that 
\begin{equation} \label{Eq:Long}
d_W \big( \pi(a), \pi(b) \big)\geq \sqrt\epsilon R.
\end{equation}
      
\subsection*{Step 1} 
Let $g \from [a,b] \to \T$ be the preferred path joining $\f(a)$ and $\f(b)$
coming from the efficiency assumption. Since $g$ is a preferred path, if  
$\omega$ is a geodesic in $\calC(W)$ joining $\pi_W(\f(a))$ and $\pi_W(\f(b))$, then 
$\omega$ can be reparametrized so that 
\[
d_W \big( g(t),\omega(t) \big) = O(1).
\]   
As stated in \eqnref{Eq:Long} we are assuming that the projection of the geodesic 
$\omega$ to $\omega_0$ has a length of at least $\sqrt\epsilon R$.
Then, by the hyperbolicity of $\calC(W)$, $\omega$ and therefore $\pi_W \circ g$ lie in a  
uniformly bounded neighborhood of $\omega_0$ along a segment of $g$ of length 
$\gmul \sqrt \epsilon R$. Divide this piece of $\omega_0$ into $3$ segments. 
Let $\alpha_1,\alpha_2, \alpha_3, \alpha_4\in \calC(W)$ be the corresponding 
endpoints of these segments which, for $i=1,2,3$, satisfy
\begin{equation} \label{Eq:Far}
d_W(\alpha_i,\alpha_{i+1})\gmul \sqrt \epsilon R.
\end{equation}
Let $\omega_{\rm mid}=[\alpha_2, \alpha_3]$ be the middle segment,
let $[c,d] \subset [a,b]$ be the associated time interval and let 
$g_{\rm mid} = g\Big|_{[c,d]}$. 

We claim that, for every $V$ so that $V \cap W \not = \emptyset$, (that is, either 
$V \subset W$ or $V \pitchfork W$) the image of the projection 
of $g_{\rm mid}$ to $\calC(V)$ has a bounded diameter. Note that, 
since $g$ is a preferred path, it is enough to prove either $d_V(\alpha_2, \alpha_3)$
or $d_V(\alpha_1, \alpha_4)$ is uniformly bounded, assuming those curves
intersect $V$. 

We argue in two cases. If every curve in $[\alpha_2, \alpha_3]$ intersects $V$,
the claim follows from the bounded geodesic image theorem 
\cite[Theorem 3.1]{minsky:CCII}. 
Otherwise, $\partial V$ is close to this segment and hence it is far from curves 
$\alpha_1$ and $\alpha_4$. Let $\bar \alpha_1$ and $\bar \alpha_4$ be curves
on $\omega_0$ that are close to $\alpha_1$ and $\alpha_4$ respectively. 
Then, $\partial V$ intersects every curve in $[\alpha_1, \bar \alpha_1]$ and
$[\alpha_4, \bar \alpha_4]$ and, by the bounded geodesic image theorem, the projections 
of these segments to $\calC(V)$ have bounded diameters. But $\omega_0$ is 
co-bounded. Hence, $d_V(\bar \alpha_1, \bar \alpha_4) =O(1)$ and therefore
$d_V(\alpha_1, \alpha_4) =O(1)$. This proves the claim.

\subsection*{Step 2}  To obtain a contradiction, we will find a large sub-box of $B$ 
that maps near a standard flat $F$ of maximal rank.
  
Note that the map $\pi$ above is quasi-Lipschitz. Choose a constant $D$ large compared 
to the quasi-Lipschitz constant of $\pi$  and the hyperbolicity constant of $\calC(W)$.
Let $a',b'$ be points in $B$ in a neighborhood of $a,b$ respectively so that
\begin{equation} \label{Eq:Near}
\|a-a'\|\leq \frac{\sqrt \ep R}{D}
\quad\text{and}\quad
\|b-b'\|\leq \frac{\sqrt \ep R}{D}.
\end{equation}
Let $g'$ be the preferred  path joining $\f(a'),\f(a')$ and $\omega'$ be the geodesic 
in $\calC(W)$ connecting $\pi_W(\f(a'))$ to $\pi_W(\f(b'))$. Consider the quadrilateral 
\[
\beta=\pi_W(\f(a)), \qquad 
\gamma=\pi_W(\f(b)) \qquad
\beta'=\pi_W(\f(a')), \qquad 
\gamma'=\pi_W(\f(b')),
\] 
in $\calC(W)$. From the assumption on $D$, the edges $[\beta, \beta']$ and 
$[\gamma, \gamma']$ are short compare to $[\beta, \gamma]$. From the 
hyperbolicity of $\calC(W)$, we conclude  that $\omega'$ has a  subsegment 
$\omega'_{\rm mid}$ that has a bounded Hausdorff distance to $\omega_{\rm mid}$.
Let $g'_{\rm mid}$ be the associated subsegment of $g'$ (see previous step). 
As we argued in the previous step, the projection of $g'_{\rm mid}$ to $\calC(V)$ has 
a bounded diameter for every $V \cap W \not = \emptyset$. In fact, it is close to the projection 
of $g_{\rm mid}$ to $\calC(V)$. 

The union of subsegments of type $[a', b']$ fill a $\frac{\sqrt \ep R}D$--neighborhood of $[c,d]$
and $|d-c| \emul \sqrt \ep R$. 
Therefore, there is a subbox $B'\subset B$ of size $R'\emul  \sqrt \epsilon R$ such 
that $(\pi_V\circ f)(B')$ has  bounded diameter for every $V \cap W \not = \emptyset$. 

For $\epsilon_0$ small to be chosen later, set $\ep_\xi=\ep_0^{6^\xi}$ and
assume $\epsilon$ is chosen so that
\[
\sqrt \epsilon < \ep_0 \epsilon_\xi.
\]
Since $f$ is $\ep$--efficient and $\ep<\ep_\xi$ it is $\epsilon_\xi$--efficient.  
By Theorem 7.2 of \cite{rafi:CD} (which can be applied if $R$ is large enough) 
there is a sub-box $B''\subset B'$ of size $R'' \geq \epsilon_\xi R'$ such that 
$\f(B'')$ is within $O(\epsilon_0 R'')$ of a standard flat $F$. The implied constants 
depend only on $K,C$ and $\xi$.  

We show this is impossible for $\epsilon_0$ sufficiently small. Note that
$\xi$ is the maximum  dimension of any standard flat. Since  $B''$ is a box 
of dimension $\xi$, and $f$ is a quasi-isometric embedding, the standard flat $F$ 
must have dimension $\xi$ as well. Let $\calV$ be the decomposition of $S$
with $|\calV| =\xi$, and let $F_V \from I_V \to \T(V)$ be the preferred paths 
where 
\[
F \from \prod_{V \in \calV} I_V \to \T.
\] 
Then $\f(B'')$ is contained in the $O(\epsilon_0 R'')$--neighborhood of 
the image of $F$. We assume $I_V$ is the smallest possible interval
for which this holds. Then, for $V \in \calV$, $F_V(I_V)$ has a diameter 
comparable to $R''$ which is the size of $B''$.

Since $|\calV| = \xi$, every $V \in \calV$ is either
an annulus or $\xi(V)=1$. Hence, they can not equal to $W$ and,
for at least one $V \in \calV$, we have $V \cap W \not = \emptyset$. In fact, we can assume 
$V$ is an annulus, because $\calV$ is maximal and if a subsurface is in 
$\calV$ the annuli associated to its boundary curves are also in $\calV$. 

From the assumption of the minimality of lengths of $I_V$, we know that 
every $t_V \in I_V$ can be completed to a vector in $\prod_{V \in \calV}I_V$
where the image is in the $O(\ep_0 R'')$--neighborhood of $\f(B'') \subset \f(B')$. 
In addition  any point in $\f(B')$ is $\ep R$ close to some $g'_{\rm mid}$. 
We already know that for any such $V$ the projection of $g'_{\rm mid}$ to $\calC(V)$ is $O(1)$. Combining these statements we find that  the projection of the image of $F$ to $\calC(V)$ of any such $V$ 
has a diameter $O(\ep R + \ep_0 R'')$. This means the same bound also 
holds for the diameter of the  projection to $\T(V)$; in the case where 
$V$ is an annulus and $\partial V$ is short, the two distances are the same. 
We have shown
\[
R'' \lmul \diam_{\T(V)}\big( F_V(I_V) \big) \lmul \ep R + \ep_0 R''.
\]
Therefore, $R'' \lmul \ep R$. But 
\[
R'' \geq \ep_\xi R' \gmul \ep_\xi \sqrt \ep R \ge \frac{\ep R}{\ep_0}.
\]
For $\epsilon_0$ sufficiently small, this is a contradiction. 
That is, the theorem holds for appropriate values of $\ep$ and $R_0$. 
\end{proof}

\subsection*{Maximal Rank is preserved}
Recall that, for $x \in \T(S)$ and a curve $\alpha$, $\tau_x(\alpha)$ is the largest number 
such that if $d(x,x')\leq \tau_x(\alpha)$ then $\Ext_{x'}(\alpha)\leq \ell_0$. 

For a point  $x$ in $\T(S)$, let $\calS_x= \calS_x(\ell_0)$ be the set of curves 
$\alpha$ such that $\Ext_x(\alpha)\leq \ell_0$. 

\begin{definition} \label{Def:Maximal_rank}
A point $x \in \T(S)$ has {\em maximal rank} when $S \setminus \calS_x$ 
are all either a pair of pants, a once punctured torus or a four-times punctured
sphere. Denote the set of points with maximal rank by $\MR$. 
Let $\LR$ its complement;  the set with  lower rank.
\end{definition}

\begin{definition} \label{Def:Adjacent} 
Suppose $x\in\MR$. 
We say a curve $\alpha\in \calS_x$ is {\em isolated} if by increasing its length to 
$\ell_0$ while keeping all other lengths the same one leaves $\MR$.  A pair of curves in $\calS_x$ 
are called {\em adjacent} if increasing both of their lengths to $\ell_0$ one leaves $\MR$. 
\end{definition}

The importance of this definition is that if $d(x,\LR)\geq d$, then every isolated curve $\alpha$ satisfies $\tau_x(\alpha)\geq d$, and for any pair of adjecent curves $\alpha_1,\alpha_2$,  at least one of which satisfies $\tau_x(\alpha_i)\geq d$. 

\begin{proposition}
\label{Prop:Rank-Preserved}
There exists $\distance_0>0$ such that, for a point $x \in \T$, 
if $d(x,\LR)\geq \distance_0$ then $f(x) \in \MR$.  
\end{proposition}

\begin{proof}
Suppose by way of contradiction that, for large $\distance_0$ 
we have a point $x$ such that $d(x,\LR)\geq \distance_0$ but 
$f(x) \in \LR$. Then there is a subsurface $W$ with $\xi(W) \geq 2$ 
so that the boundary curves of $W$ are $\ell_0$ short but no curve in $W$ 
is shorter than $\ell_0$ in $f(x)$. 

Let $g_0$ be a path passing through $f(x)$ that is $\cb$--cobounded relative 
to $W$ (see the discussion after \defref{Def:CB}). Let $f(y)$ be a point in $g_0$ 
so that the distance in $\T$ between $f(x)$ and $f(y)$ is $L=\distance_0/2\mul-\add$.  
Since $g_0$ is $\cb$ cobounded,  we have from \eqnref{eq:distance}
\begin{equation} \label{Eq:W-Progress}
d_W\big( f(x), f(y) \big) \gmul L. 
\end{equation}
Since $f$ is a $(\mul,\add)$--quasi-isometry, solving \eqnref{Eq:QI} 
for $d(x,y)$ we get 
\[
\frac{\distance_0}{2\mul^2} - \frac{2\add}{\mul} \leq d(x, y) \leq \frac{\distance_0} 2.
\] 
In particular, $y\in \MR$.  

Note that, in addition, for $\distance_0\geq 2\Log \frac 1{\ell_0}$,  there is a 
maximal product region $\T_\calU$ containing both $x$ and $y$. 
For $d_\T(x, \LR) \geq \distance_0$ implies that there is a set of curves
$\balpha$, where $\tau_x(\alpha) \geq \distance_0$ for $\alpha \in \balpha$,
and so that  the complementary regions have complexity at most one. 
For all these curves, we have $\tau_y(\alpha) \geq \distance_0/2$; in particular 
they are at least $\ell_0$ short in $y$. This means, both $x$ and $y$ are in 
$\T_\calU$. 

In fact, there is a box 
\[
B = \prod_{U \in \calU} I_U \subset \R^\xi
\] 
of size $L$ and a quasi-isometry 
\[
Q = \prod_{U \in \calU} Q_U \from B \to \T
\] 
where each $Q_U \from I_U \to \T(U)$ is a geodesic and $x$ and $y$ are contained
in $Q(\underline B)$ (recall that $\underline B$ is the central sub-box of half 
the diameter). The map $Q$ is a quasi-isometry because $B$ is equipped with 
the Euclidean metric and $T_\calU$ is equipped with the sup metric up to an
additive error of $\Distance_0$. Define
\[
\f \from B \to \T, \qquad\text{by}\qquad \f = f \circ Q. 
\]
Then $\f$ is a $(K,C)$-quasi-isometric embedding where $K$ and $C$
depend on $\mul$, $\add$ and the complexity $\xi= |\calU|$. 
($\Distance_0$ depends only on these constants.)

Let $\omega_0$ be the geodesic in $\calC(W)$ that shadows the projection
$\pi_U(g_0)$. Then $\omega_0$ is $M$-cobounded with $M$ slightly larger than $\cb$. 
Define  
\[
\pi=\pi_{\omega_0} \circ \pi_U \circ \f.
\]
and let $l_\pi \emul \mul$ be the Lipschitz constant of $\pi$. 

Let $\ep$ and $R_0$ be constants from \thmref{Thm:Tunnel} associated to 
$K$, $C$ and $M$ and chose $\theta$ so that $\sqrt[\xi]{\theta} \, l_\pi$ is small
(see below). Then, let $L_0$ be the constant given by 
\thmref{Thm:Differentiable} (the dimension $n$ equals $\xi$). Choose 
$\distance_0$ large enough so that 
\[
L = d(x,y)\geq  \frac{\distance_0}{2\mul^2}-\frac{2 \add}\mul \geq L_0.
\] 
Applying \thmref{Thm:Differentiable} to $B$, we conclude that there is scale $R$
and a decomposition $\calB_R$ of $\underline B$ to boxes of size $R$ so that 
a proportion at least $(1-\theta)$ of boxes in $\calB_R$ are $\ep$--efficient.  

By \lemref{Lem:Path}, there exists a path $\gamma$ joining $x$ to $y$ that 
is covered by at most $O(L/R)$ boxes in $\calB_R$ of which at most
$O\left( \sqrt[\xi]{\theta} \frac{L}{R}\right)$ are not $\epsilon$--efficient.  

Assume $\gamma$ intersect boxes $B_1, \ldots, B_k$ and let
$\gamma_i$ be the subinterval of $\gamma$ associated to $B_i$. 
By the triangle inequality, the sum of the diameters of $\pi(\gamma_i)$ is 
larger than $d_W \big( f(x), f(y) \big)$. However, by \thmref{Thm:Tunnel}, 
\begin{equation} \label{Eq:Efficient} 
\sum_{B_i \ \text{is efficient}} \diam_{\calC(W)} \pi(\gamma_i) 
\lmul  \frac{L}{R} \sqrt \ep R = \sqrt{\ep} \, L.
\end{equation}
And the assumption on the number of non-efficient boxes gives
\begin{equation} \label{Eq:Non-Efficient} 
\sum_{B_i \ \text{is not efficient}} \diam_{\calC(W)} \pi(\gamma_i) 
\lmul \left( \sqrt[\xi]{\theta} \frac{L}{R} \right) l_\pi \, R = 
\sqrt[\xi]{\theta} \, l_\pi \, L. 
\end{equation}
For $\ep$ and $\theta$ small enough, Equations \eqref{Eq:Efficient} 
and \eqref{Eq:Non-Efficient} contradict \eqnref{Eq:W-Progress}. 
This finishes the proof. 
\end{proof}

\section{Local Splitting Theorem}
\label{Sec:Local}
In this section, we prove a local version of splitting theorem
proven by Kleiner-Leeb \cite{kleiner:RS}  and Eskin-Farb \cite{eskin:QF}.

\begin{theorem} \label{Thm:Factor-Preserving}
For every $K, C, \bar\rho$ there are constants $R_0$, $D$ and $\rho$ such that, for all 
\[
\bz=(z_1, \ldots, z_m) \in \prod_{i=1}^m \H_i
\] 
and $R \geq R_0$ the following holds. Let $B_R(\bz)$ be a ball of radius 
$R$ centered at $\bz$ in $\prod_{i=1}^m \H_i$ and let 
$\f \from B_R(\bz) \to \prod_{i=1}^m \H_i$ be a 
$(K, C)$--quasi-isometric embedding whose image coarsely contains  a ball of radius 
$\bar\rho R$ about $\bar f(\bz)$.   Then there is a smaller ball $B_{\rho R}(\bz)$,
a permutation $\sigma \from \{1, 2, \ldots, m\} \to \{1, 2, \ldots, m\}$ and
$(K, C)$--quasi isometric embeddings  
\[
\phi_i \from B_{\rho R}(z_i) \to \H_{\sigma(i)},
\]
so that the restriction of $\f$ to  $B_{\rho R}(\bz)$ is $D$--close to  
\[
\phi_1 \times \ldots \times \phi_m \from 
B_{\rho R} (\bz) \to \prod_{i=1}^m \H_{\sigma(i)}.
\]
\end{theorem}

\begin{remark} \label{Rem:Distance}
When we use this theorem in \secref{Sec:Local-Factors}, we need to equip 
$\prod \H$ with the $L^\infty$--metric. However, it is more convenient to use the 
$L^2$--metric for the proof. Note that, if $\f$ is a quasi-isometry with respect to one 
metric, it is also a quasi-isometry with respect to the other. For the rest of 
this section, we assume $\prod \H$ is equipped with the $L^2$--metric. 
To simplify notation, we use $d_\H$ to denote the distance in both in $\H$ and in 
$\prod \H$. 
\end{remark} 

For the proof, we will use the notion of an asymptotic cone. Our brief
discussion is taken from \cite{kleiner:RS}.  
A non-principal ultrafilter is a finitely additive probability measure
$\omega$ on the subsets of the natural numbers $\N$ such that
\begin{itemize}
\item $\omega(S)=0$ or $1$ for every $S\subset \N$ 
\item $\omega(S)=0$ for every finite subset $S\subset \N$
\end{itemize}
Given a bounded sequence $\{a_n\}$ in $\R$, there is a unique limit point 
$a_\omega \in \R$ such that, for every neighborhood $U$ of $a_\omega$, the set
$\{ n \st a_n \in U\}$ has full $\omega$ measure. We write 
$a_\omega = \olim a_n$. 

Let $(\calX_n,d_n,*_n)$ a sequence  of metric spaces with base-points. 
Consider
\[
\calX_\infty=\Big\{ \vec x = (x_1, x_2, \ldots ) \in \prod \calX_i:d(x_i,*_i) \text{is bounded}\Big\}.
\]
Define $\bar d_\omega:\calX_\infty \times \calX_\infty \to \R$ by 
\[
\bar d_\omega(\vec x, \vec y)=\olim d_i(x_i,y_i).
\]
Now $\bar d_\omega$ is a pseudo-distance. Define  the ultralimit of the sequence 
$(\calX_n,d_n,*_n)$ to be the quotient metric space  $(\calX_\omega,d_\omega)$ 
identifying the points of distance zero. 

Let $\calX$ be a metric space and $*$ be a basepoint. 
The asymptotic cone of $\calX$, $\Cone(\calX)$, 
with respect to the non-principal ultrafilter $\omega$ and the sequence $\lambda_n$ 
of scale factors with $\olim\lambda_n=\infty$ and the basepoint $*$,
is defined to be the ultralimit of the sequence of rescaled spaces
$(\calX_n,d_n,*_n):= (\calX,\frac{1}{\lambda_n}d_n,*)$. The asymptotic cone
is independant of the basepoint. 

In the case of $\H$, the asymptotic cone $\H_\omega$ is a metric tree which
branches at every point and the asymptotic cone 
$\big( \prod_{i=1}^m\H\big)_\omega$ of the product of hyperbolic planes is  
$\prod_{i=1}^m\H_\omega$, the product of the asymptotic cones. 
A flat in $\prod_{i=1}^m \H$ is a product $\prod_{i=1}^m g_i$ where $g_i$ 
is a geodesic in the $i^{th}$ factor.  

We first prove a version of \thmref{Thm:Factor-Preserving} with small linear 
error term. We then show that, by taking an even smaller ball, the error term
can be made to be uniform additive. 

\begin{proposition} \label{Prop:Linear-Factoring}
Given $K,C,\bar\rho$ there exists $\rho'>0$ and $D_0$ such that for all sufficiently small  $\epsilon>0$, 
there exists $R_0$ such that if $R\geq R_0$ and $\f$ is a $(K,C)$--quasi-isometric  
embedding defined on $B_R(\bz)$, such that 
$\f(B_R(\bz))$  $C$--coarsely contains $B_{\bar\rho R}(\f(\bz))$ then 
\begin{itemize}
\item There is a permutation $\sigma$ and, for $1 \leq i \leq m$, there is a 
quasi-isometric embedding $\phi_i^{\bz} \from B_{\rho' R}(z_i) \to \H_{\sigma(i)}$ 
so that, for 
\[
\phi^{\bz}=\phi_1^{\bz} \times \ldots \times \phi^{\bz}_m \from 
B_{\rho' R}(\bz) \to \prod_{i=1}^m \H,
\] 
and for $\bx \in B_{\rho' R}(\bz)$, we have 
\begin{equation} \label{Eq:Linear-Error} 
d_{\H}\big( \, \f(\bx),\phi^{\bz}(\bx) \big)\leq \epsilon d_\H(\bz,\bx)+D_0.
\end{equation}
\item For any  $\bx \in B_{\rho' R}(\bz)$ and any flat $F_\bx$ through $\bx$, there is 
a flat $F_\bx'$ such that for $\bp\in N_{\rho' R}(\bx)\cap F_\bx$
\[
d_\H \big( \, \f(\bp), F_\bx' \big)\leq \epsilon d_\H(\bx,\bp)+D_0.
\]
\end{itemize}
\end{proposition}

\begin{proof}
We begin with a claim. 

\subsection*{Claim} Assume, for given $\ep$ and $D_0$, that there is a permutation 
$\sigma$ and a constant $\rho'$ so that, if $\bx, \by \in B_{\rho' R}(\bz)$ differ only in 
the $i^{th}$ factor, then $\f(\bx)$ and $\f(\by)$ differ in all factors besides the 
$\sigma(i)^{th}$ factor by at most $\frac{\ep d_\H(\by, \bx) +D_0}m$. Then the first 
conclusion holds. 

\begin{proof}[Proof of Claim] \renewcommand{\qedsymbol}{$\blacksquare$} 
For $x \in B_{\rho' R}(z_i)$ and an index $i$ define $\bz^i_{x}$ 
to be a point whose $i^{th}$ coordinate is $x$ and whose other coordinates are the 
same as the coordinates of $\bz$. We then define the map $\phi_i^{\bz}$ by letting 
$\phi_i^{\bz}(x)$ to be the $\sigma(i)^{th}$ coordinate of $\f(\bz^i_{x})$.  
We show that $\phi_i^{\bz}$ is a quasi-isometric embedding. 
For $x, x' \in B_{\rho' R}(z_i)$ 
\[
d_\H \big( \phi_i^\bz(x),  \phi_i^\bz(x')  \big) \leq
d_\H \big( \f (\bz^i_{x}) ,   \f (\bz^i_{x'}) \big) \leq 
K d_\H \big(\bz^i_{x}, \bz^i_{x'} \big) + C\leq
K d_\H \big(x, x' \big) + C.
\]
In addition since $\bz^i_{x}$ and $\bz^i_{x'}$ differ in only one factor, 
for $\ep \leq \frac 1{2K}$, 
\begin{align*}
d_\H \big( \phi_i^\bz(x),  \phi_i^\bz(x')  \big) 
 & \geq d_\H \big( \f (\bz^i_x),\f (\bz^i_{x'}) \big)
               -(m-1)\frac{\ep d_\H(\bz^i_x, \bz^i_{x'}) +D_0}m\\
 & \geq \frac{d_\H \big(\bz^i_{x}, \bz^i_{x'} \big)}K  - C 
               - \ep d_\H \big(\bz^i_{x}, \bz^i_{x'}\big) - D_0\\
  & \geq \frac{d_\H(x, x')}{2K} -  (C + D_0).       
\end{align*}

Hence, $\phi_i^{\bz}$ is a $(2K, C + D_0)$--quasi-isometry. 
\eqnref{Eq:Linear-Error} follows from applying the triangle inequality $m$--times. 
\end{proof} 

Now, suppose the first conclusion is false. Then there exists $K,C,\epsilon>0$, sequences 
$\rho_n\to 0, D_n\to\infty$,  and a sequence $\f_n$ of $(K,C)$--quasi-isometric 
embeddings defined on the balls $B_{R_n}(\bz_n)$, with $R_n \to \infty$, so that the 
restriction of $\f_n$ to $B_{\rho_n R_n}(\bz_n)$ does not factor as above. Then, by 
the above claim, there exists points $\bx_n, \by_n \in B_{\rho_n R_n}(\bz_n)$ 
which differ in one factor only, and such that $\f_n(\bx_n)$ and $\f_n(\by_n)$ differ in at 
least two factors by an amount that is at least $\frac{\epsilon d_\H(\by_n,\bx_n)+D_n}m$ 
in each.   We can   assume $d_\H( \by_n,\bx_n )\to\infty$ for otherwise  
$d_\H\big( \, \f(\by_n),\f(\bx_n)\big)$ is bounded.   

Let $\lambda_n=\frac{1}{d_\H(\by_n,\bx_n)}$ and scale the metric on $B_{R_n}(\bz_n)$ 
with base-point $\bz_n$ by $\lambda_n$.  Since $\rho_n \to 0$, we have 
$\lambda_n R_n \geq \frac 2{\rho_n} \to \infty$. That is, the radius of $B_{R_n}(\bz_n)$ in 
the scaled metric still goes to $\infty$. However in the scaled metric, the distance between 
$\bx_n$ and $\by_n$ equals $1$. Let $\prod_{i=1}^m \H_\omega$ be the 
the asymptotic cone of $\prod_{i=1}^m \H$ with base point $\bz_n$ and metric
$d_n = \lambda_n d_\H$. For any $(\bu_1,\bu_2\ldots )\in  \prod_{i=1}^m\H_\omega$,  
We define 
\[
\f_\omega\from \prod_{i=1}^m \H_\omega\to \prod_{i=1}^m \H_\omega.
\] 
by
\[
\f_\omega(\bu_1,\bu_2,\ldots)=\big( \, \f_1(\bu_1),\f_2(\bu_2),\ldots \big).
\]
Note that, since by definition $\lambda_n d_\H(\bu_n,\bz_n)$ is a bounded, 
for $n$ large enough, $\bu_n \in B_{R_n}(\bz_n)$ and $f_n(\bu_n)$ is defined. 
It is clear that $f_\omega$ is bi-Lipschitz. We show $\f_\omega$ is onto. 

By assumption $\f_n(B_{R_n}(\bz_n))$ $C$--coarsely contains 
$B_{\bar \rho R_n}(\f(\bz_n))$. Consider a point 
\[
(\bw_1,\bw_2,\ldots)\in \prod_{i=1}^m \H_\omega.
\]
Since $\lambda_n d_\H(\bz_n, \bw_n)$ is bounded, for $n$ large enough, 
$\bw_n \in B_{\bar \rho R_n}(\bz_n)$. This  means  there is $\bu_n \in B_{R_n}(\bz_n)$  
so that 
\[
d_\H \big(\, \f_n(\bu_n), \bw_n \big)  \leq C. 
\]
But $\lambda_n \to \infty$. Thus 
\[
(\bw_1,\bw_2,\ldots )=\f_\omega(\bu_1,\bu_2,\ldots)
\]
and so $\f_\omega$ is onto and hence a homeomorphism. 

By the argument in Step 3 of Section 9 in \cite{kleiner:RS} the map $\f_\omega$  factors.
The $\omega$-limit points of $\bx_n,\by_n$ give a pair of points  $\bx_\omega$ and  $\by_\omega$  
in $\prod_{i=1}^m \H_\omega$ that have the same coordinate in every factor but one and
\[
d_{\H_\omega} (\bx_\omega,\by_\omega)=1.
\]
But $\f_\omega(\bx_\omega)$ and $\f_\omega(\by_\omega)$ differ in at least two 
coordinates by at least $\frac\ep m$. This contradicts the assumption that $\f_\omega$ 
factors. 

We now use the first conclusion to prove the second conclusion. Let $r=\rho' R$. 
Consider a flat $F_\bx$ through $\bx$ and let $g_i = [a_i, b_i]$ be a geodesic in the 
$i^{th}$  factor so that
\[
F_\bx \cap B_{r}(\bx) \subset \prod_{i=1}^m [a_i,b_i]
\]  
Let $g_i'$ be the geodesic joining $\phi_i^{\bz}(a_i)$ to $\phi_i^{\bz}(b_i)$.   Since 
$\phi_i^{\bz}$ is a quasi-isometric embedding,    
\[
d_\H\big( \phi_i^{\bz}(g_i),g_i' \big)=O(1),
\]
where the bound depends on $K,C$.  Let $F_\bx'$ be the flat determined by the $g_i'$.  
For a point $\bp\in F\cap B_{r}(\bx)$, the $i^{th}$ coordinate of $\bp$ lies on $g_i$. 
Therefore, the $i^{th}$ coordinate of $\f(\bp)$ is distance at most $\epsilon d(\bx,\bp)$ 
from a point whose $i^{th}$ coordinate lies on $\phi_i^{\bz}(g_i)$ and that in turn is distance
$O(1)$ from $g_i'$. Since this is true for each $i$, the triangle inequality  implies that  
\begin{equation*}
d_\H \big(\, \f(\bp),F' \big) \leq \epsilon d_\H(\bx,\bp)+O(1). \qedhere 
\end{equation*}   
\end{proof}

\begin{lemma}
\label{lemma:trivial:flats}
Fix  a constant $\rho''<1$.  There exists $D''$, such that for all sufficiently small $\epsilon$ 
and large $R'$ the following holds. Suppose $F$, $F'$ are flats, $\bp \in F$ and 
\begin{displaymath}
B_{R'}(\bp) \cap F\subset \calN_{\ep R'}(F'). 
\end{displaymath}
Then for  $r\leq  \rho''R'$,
\begin{displaymath}
B_r(\bp) \cap F\subset  \calN_{D''}(F'). 
\end{displaymath}
\end{lemma}

\begin{proof}
We can assume $\epsilon R'$ is larger than the hyperbolicity constant for $\H$.  
Consider any geodesic    $\gamma\subset F$ whose projection to each factor $\H$ 
intersects the disc of radius $r<\rho''R'$ centered at the projection of $\bp$ to that factor. 
Extend the geodesic  so that its endpoints lie  further than $\epsilon R'$ from the disc.  
Choose a pair of points in $F'$ within $\epsilon R'$ of the endpoints of $\gamma$ and 
let $\gamma'$ be the geodesic in $F'$ joining these points.  Then  $\gamma'$ lies within 
Hausdorff distance $\epsilon R'$ of $\gamma$. Form the quadrilateral with two additional 
segments joining the endpoints of $\gamma$ and $\gamma'$. Since the  endpoints are 
within $\epsilon R'$ of each other, the segments joining the endpoints do not enter the 
disc of radius $r$. The quadrilateral is $2\delta$ thin, where $\delta$ is the hyperbolicity 
constant for $\H$.  Therefore   $\gamma$ and $\gamma'$ are within Hausdorff distance 
$O(2\delta)$ of each other on the disc of radius $r$.   
\end{proof}

\begin{proof}[Proof of \thmref{Thm:Factor-Preserving}.]
By the second conclusion of \propref{Prop:Linear-Factoring}  there exists $\rho',D_0$ 
so that for all small $\ep$ and large $R$, for any  $\bx \in B_{\rho' R}(\bz)$ and any flat 
$F$ through $\bx$, there is a flat $F_{\bx}'$ such that for $\bp\in F\cap B_{\rho' R}(\bx)$
\[
d_\H \big( \, \f(\bp), F_{\bx}' \big)\leq \epsilon d(\bx,\bp)+D_0.
\]
Now clearly for all $\bx$, $d(\f(\bx),F_{\bx}')\leq D_0$. 
By \lemref{lemma:trivial:flats} there is $\bar D=\bar D(D_0,D'')$ and $\rho''$, such that for 
$\epsilon$ sufficiently small and $R$ large, and any pair of points 
\[
\bx,\bp\in F\cap B_{\rho''\rho'R}(\bz),
\] 
the flats $F_{\bp}',F_{\bx}'$ corresponding to $\bp$ and $\bx$ 
satisfying the above inequality,   are within $\bar D$ of each
other. That is,  given $F$ there is a single flat $F'$ so that for all 
$\bx\in F\cap B_{\rho''\rho' R}(\bz)$,  
\[
d_\H \big(\, \f(\bx),F' \big)\leq \bar D.
\]
 
Consider any  geodesic $g_i$ in the $i^{th}$ factor of $\prod_{i=1}^m\H$ that intersects 
$B_{\rho''\rho' R}(z_i)$ and fix the other coordinates so that we have   a geodesic in 
$\prod_{i=1}^m \H$ that intersects $B_{\rho''\rho' R}(\bz)$.  Again denote it by  $g_i$.   
Choose a pair of flats $F_1,F_2$ that intersect exactly along $g_i$. Then as we have 
seen there are flats $F_1',F_2'$  such that for $j=1,2$, 
\[
\f \big( F_j\cap B_{\rho''\rho' R}(\bz_0) \big) \subset \calN_{\bar D} (F_j')
\] 
and thus for both $j=1,2$,  
\[
\f \big( g_i\cap B_{\rho''\rho' R}(\bz_0) \big)\subset \calN_{\bar D}(F_j').
\]
Since $\f$ is a quasi-isometric embedding  the pair  of flats $F_1',F_2'$ must come 
$O(\bar D)$ close along a single geodesic of length  comparable to $R$ in  one factor 
in  each. Thus we can assume that $\f$ factors along $g_i$ and sends its intersection 
with $B_{\rho''\rho' R}(\bz_0)$  to within $O(\bar D)$ of a geodesic $g_j'$ in a factor $j$.   

Now let $\rho'''<1$ and set $\rho=\rho'''\rho''\rho'$.  
Now consider  {\em any}  geodesic  $g$ in the $i^{th}$  factor that intersects the smaller 
ball $B_{\rho R}(\bz_0)$.  We have that  $\f$ factors  in the bigger ball 
$B_{\rho''\rho' R}(\bz_0)$ along   $g\cap B_{\rho''\rho' R}(\bz_0)$.  We claim that it 
also sends it also to the same $j^{th}$ factor.  Suppose not, and it sends it to the 
$k\neq j$ factor.  Choose  a geodesic $\ell$ that comes close to both $g_i$ and $g$ 
possibly in the bigger ball $B_{\rho''\rho' R}(\bz_0)$. Its image  must change from the 
$j^{th}$ to the $k^{th}$ factor, which is impossible since the image of 
$\ell\cap B_{\rho''\rho' R}(\bz_0)$ lies in a single factor up to bounded error.    
Thus the map is factor preserving. 
\end{proof}

\section{Local Factors in Teichm\"uller space} \label{Sec:Local-Factors} 
In this section we apply \thmref{Thm:Factor-Preserving} to balls in \Teich space.

For a given $R>0$ and $x \in \T$, define the $R$--decomposition 
at $x$ to be the decomposition $\calU$ that contains a curve $\alpha$
if and only if $\tau_x(\alpha) \geq R$. That is, elements of $\calU$ are
either such curves or their complementary components. 
By convention, if $\T = \T(\Sigma, L)$ and $\Sigma$ has a component
$U$ with $\xi(U)=1$ then $U$ (not any annulus in $U$) is always included in any 
$R$--decomposition of $\Sigma$. This is because $\T(U)$ is already a copy 
of $\H$. An $R$-decomposition is maximal if there are no complimentary components $W$ with $\xi(W)>1$.  We always assume $f \from \T \to \T$ is a $(\mul, \add)$--quasi-isometry
but, unless specified, it is not always assumed that $f$ is anchored.

\begin{proposition} \label{Prop:Local-Factors}
For $\mul$ and $\add$ as before, there are constants $0<\rho_1 <1$, 
$\distance_1, \add_1$ and $\Distance_1$ so that the following holds. 
For $x \in \MR$ let $R$ be such that 
\[
d(x, \LR) \geq R \geq \distance_1.
\]
Let $\calU$ be the $R$--decomposition at $x$, let $\calV$ be the $\frac{R}{2\mul}$--decomposition at $f(x)$ and let $r = \rho_1 R$. We have
\begin{enumerate}
\item The decompositions $\calU$ and $\calV$ are maximal. 
For $x'\in B_r(x)$, we have $x' \in \T_\calU$ and $f(x') \in \T_\calV$.
\item There is a bijection $f^\star_x \from \calU \to \calV$ and, for every 
$U \in \calU$ and $V=  f^\star_x(U)$, there is a $(\mul, \add_1)$--quasi-isometry 
\[
\phi_x^U \from B_r(x_U) \to \T\big(V \big),
\]
so that, for all $x'\in B_r(x)$,
\[
d_{\T(V)} \Big( \phi_x^U(x'_U), f(x')_V \Big) \leq \Distance_1.
\]
\end{enumerate}
\end{proposition}

\begin{remark}
Condition (2) above states that the map $f$ restricted to $B_r(x)$ is close to the 
product map $\prod_{U \in \, \calU} \phi_x^U$. Note also that, since $\calU$
depend on $R$, $\fs_x$ also depends on $R$. 
\end{remark}

\begin{proof}
Let $\balpha$ be the set of curves $\alpha$ with $\tau_X(\alpha)\geq R$.
Since $d_\T(x,\LR) \geq R$ any complementary component $U$ of $\balpha$ must satisfy 
$\xi(U)\leq 1$. Otherwise there would be a pair of adjacent curves 
$\gamma_1,\gamma_2\subset U$ with $\tau_x(\gamma_i)< r$, which means 
that $d_\T(x,\LR)<r$.  We conclude that $\calU$ is a maximal decomposition. 
 
Let $ \rho_0 = \frac{1}{5 \mul^2}$. For $x'\in B_{\rho_0 R}(x)$, we have 
\[
\tau_{x'}(\alpha)\geq \tau_{x}(\alpha) - \rho_0 R \geq (1-\rho_0)R.
\]  
This, if $\distance_1$ is large enough, implies that the curves $\alpha\in \balpha$ are 
$\ell_0$--short in $x'$ and so $x' \in \T_\calU$.

Let $y \in \LR$ be a point such that 
\[
d_\T\big( f(x),y \big)=d_\T\big( f(x),\LR \big).
\] 
By \propref{Prop:Rank-Preserved}, 
\[
d_\T \big(f^{-1}(y), \LR \big) \leq \distance_0.
\] 
Using first the triangle inequality and then the fact that $f^{-1}$ is a 
$(\mul,\add)$--quasi-isometry, we get
\[
R=d_\T(x,\LR)\leq d_\T \big(x,f^{-1}(y)\big)+\distance_0
    \leq \mul d_\T\big(f(x),\LR\big)+\add+\distance_0.
\] 
By picking  $\distance_1$ large enough in terms of $\distance_0,\mul,\add$, we have 
\begin{equation} \label{Eq:2K}
d_\T \big( f(x),\LR \big) \geq \frac{R}{2\mul}.
\end{equation}
This means, as argued above, that if $\bbeta$ is the collection of curves $\beta$
with $\tau_{f(x)}(\beta)\geq \frac{R}{2\mul}$, then any complementary components 
$V$ satisfies $\xi(V)\leq 1$. Hence $\calV$ is a maximal decomposition. 

Again, since $f$ is a  $(\mul,\add)$--quasi-isometry, for all $\beta \in \bbeta$, we have 
\begin{align*}
|\tau_{f(x)}(\beta)-\tau_{f(x')}(\beta)| & \leq d(f(x),f(x')) \leq \mul d(x,x')+\add \\
  &\leq \mul (\rho_0 R)+\add\leq \frac{R}{5\mul}+\add\leq \frac{R}{4\mul}.
\end{align*}
The last inequality holds for $\distance_1$ large enough. Thus, 
for $x'\in B_{\rho_0 R} (x)$ and $\beta\in\bbeta$,
\[
\tau_{f(x')}(\beta)\geq \tau_{f(x)}(\beta)-\frac{R}{4\mul}
      \geq \frac{R}{2\mul}-\frac{R}{4\mul}.
\] 
Again, for $\distance_1$ large enough, this means $\beta$ is $\ell_0$--short
and thus $f(x')\in \T_{\calV}$.  

We have shown that $x'\in B_{\rho_0 R}(x)$ implies $x'\in\T_{\calU}$ and $f(x')\in\T_\calV$. 
By the Minsky Product Region Theorem (\thmref{Thm:Product-Regions}) 
the maps $\psi_\calU$ and $\psi_\calV$ are distance $\Distance_0$ from an isometry. 
Define 
\[
\f \from \prod_{U\in \, \calU}B_{\rho_0 R}(x_U) \to \prod_{V\in \calV} \T(V),
\]
by
\[
\f = \psi_\calV \circ f \circ \psi_\calU^{-1}. 
\]
Then, if we set $\add_1=2\Distance_0+\add$, the map $\f$ is a 
$(\mul,\add_1)$--quasi-isometry.   

Because the map $f$ has an inverse,  $\f(B_{\rho_0 R}(\bx)$ contains a ball of comparable radius about $\f(\bx)$. 
Now \thmref{Thm:Factor-Preserving} applied for $K = \mul$ and $C = \add_1$
says that, there are constants $R_0$, $\rho$, $D$ and a bijection $f_x^* \from \calU\to\calV$ 
so that, for $r = \rho \, \rho_0 R$ the following holds. Assume $\distance_1 \geq R_0$. 
Then, for each $U\in\calU$ and  $V = f_x^*(U)$, there is a $(\mul,\add_1)$--quasi-isometry  
\[
\phi_x^U \from B_r(x_U)\to \T\big(V\big)
\]
such that  for $x'\in B_r(x)$,  
\[
d_{\T_\calV} \left( \f \big( \psi_{\calU}(x') \big), 
        \prod_{U\in \calU}\phi_x^U(x'_U) \right)\leq D.
\]
The distance in $\T_\calV$ is the sup metric, each factor of which is 
either a copy of $\H$ or a horosphere $H_\beta \subset\H$.  
Hence, the inequality holds for every factor. 
\[
d_{\T(V)} \Big( f\big( \psi_{\calU}(x') \big)_V, 
       \phi_x^U(x'_U) \Big)\leq D.
\]
Therefore, for $\rho_1 = \rho \, \rho_0$, $\Distance_1 = D$ and $\distance_1$ 
large enough, the proposition holds.
\end{proof}

\begin{proposition} \label{Prop:Analytic-Continuation} 
Choose $R$ so that 
\[
r = \rho_1 R \geq \max \big( \rho_1\distance_1, 4\mul(4 \Distance_1 + \add_1) \big).
\]  
Let $x^1, x^2 \in \T$ be points so that 
\[
\d_\T\left(x^1, \LR\right) \geq R, \quad
\d_\T\left(x^2, \LR\right) \geq R
\quad\text{and}\quad
d_\T\left(x^1, x^2\right) \leq r.
\]
Assume every point $x \in B_r(x^1) \cup B_r(x^2)$ has the same 
$R$--decomposition $\calU$ and the $\frac{R}{2\mul}$--decomposition 
at $f(x)$ always contains some subsurface $V$. 
Then, there is $U \in \calU$ so that 
\[
f_{x^1}^\star(U) = f_{x^2}^\star(U) = V,
\]
and, for every $u \in B_r(x^1_U) \cap B_r(x^2_U)$
\[
 d_{\T(V)}\Big( \phi_{x^1}^U(u), \phi_{x^2}^U(u) \Big) \leq 2 \Distance_1. 
\]
\end{proposition}

\begin{remark}
Since $f_{x^i}^*$ is a bijection there must be some $U$ that is mapped to  $V$. 
The content of the first conclusion is that the same $U$ works at both points. 
\end{remark}

\begin{proof}
By assumption, for $x \in B_r(x^1) \cap B_r(x^2)$, the domain of $\fs_x$ is $\calU$. 
We start by proving the following claim. For $z^1, z^2 \in B_r(x^1) \cap B_r(x^2)$ 
and $U \in \calU$, suppose $\fs_{z^1}(U)=V$. Also, assume either
\[
d_{\T(U)}\left( z^1_U,z^2_U \right) \geq \frac{r}{4}
\qquad\text{and}\qquad
 \forall \, W \in \calU- \{U\} \quad z^1_{W}=z^2_{W},
\]
or 
\[
 \forall \, W \in\calU-\{U\} \quad d_{\T(W)}\left( z^1_{W},z^2_{W} \right)  \geq \frac{r}{4}
\qquad\text{and}\qquad
z^1_{U}=z^2_{U}.
\]
That is, $z^1$ and $z^2$ either differ by $r/4$ in only one factor, or all but one factor. 
Then $\fs_{z^2}(U)=V$.
  
We prove the claim.  Assume the first case holds. Since $\phi_{z^1}^U$ is a 
$(\mul, \add_1)$--quasi-isometry  (\propref{Prop:Local-Factors}) we have  
\[
d_{\T(V)}\Big(\phi_{z^1}^U(z^1),\phi_{z^1}^U(z^2) \Big)
    \geq \frac{r}{4\mul} - \add_1.
\]
By \propref{Prop:Local-Factors} and the triangle inequality applied twice, we get 
\[
d_{\T(V)}\Big( f(z^1)_V , f(z^2)_V \Big)
  \geq \frac{r}{4\mul}-\add_1-2\Distance_1.
\]

Now suppose  $\fs_{z^2}(W)=V$ for $W \neq U$.  Since $z^1_{W}=z^2_{W}$,
\propref{Prop:Local-Factors}  and the triangle inequality then implies 
\[
d_{\T(V)}\Big( f(z^1)_{V},f(z^2)_{V} \Big)\leq 2\Distance_1.
\]
These two inequalities contradict the choice of $R$ in the statement of the 
Proposition, proving the the claim. The proof of the claim when the second 
assumption holds is similar. 

We now prove the Proposition. In each $W \in \calU$ choose $z_W$
so that 
\[
d_{\T(W)}(z_W, x^1_W) =  d_{\T(W)}(z_W, x^2_W) = \frac{r}4.
\]
Let $z^1$, $z^{1,2}$ and $z^2$ be points in $B_r(x^1) \cap B_r(x^2)$ so that,
for $i=1,2$
\[
z^i_U=z_U, 
\qquad\text{and}\ 
\forall\, W \in \calU-\{U\} \quad z^i_W= x^i_W
\]
and 
\[
\forall\, W \in \calU \quad z^{1,2}_W = z_W.
\]
Note that the claim can be applies to pairs $(x^1, z^1)$, 
$(z^1, z^{1,2})$, $(z^{1,2}, z^2)$ and $(z^2, x^2)$ concluding that
\[
V= \fs_{x^1}(U) = \fs_{z^1}(U) = \fs_{z^{1,2}}(U) = \fs_{z^2}(U) = \fs_{x^2}(U).
\]

Now, consider $u \in B_r(x^1_U) \cap B_r(x^2_U)$ and let 
$z \in B_r(x^1) \cap B_r(x^2)$ be so that $z_U = u$. We know
\[
d_{\T(V)} \Big( f(z)_V, \phi_{x^1}^U(u) \Big) \leq \Distance_1
\qquad\text{and}\qquad
d_{\T(V)} \Big( f(z)_V, \phi_{x^2}^U(u) \Big) \leq \Distance_1.
\]
Therefore, 
\[
d_{\T(V)} \Big( \phi_{x^1}^U(u), \phi_{x^2}^U(u) \Big) \leq 2\Distance_1.
\]
This finishes the proof of the proposition. 
\end{proof}

\begin{corollary} \label{Cor:xi=1}
If $\T= \T(\Sigma, L)$ and $U$ is a component of $\Sigma$ 
with $\xi(U) =1$, then for every $x \in \T(\Sigma, L)$ where $\fs_x$ is defined, 
$\fs_x(U)=U$. 
\end{corollary}

\begin{proof}
Note that, by definition, the subsurface $U$ is always included in any 
$R$--decomposition $\calU$ of $\Sigma$. Hence, \propref{Prop:Analytic-Continuation}  
applies. That is, for a large $R$ as in \propref{Prop:Analytic-Continuation},
if $\fs_x(U)=U$ then the same hold for points in an $r$--neighborhood of $x$. 
But the set of points where $\fs_x$ is defined is connected. Hence, it is enough to 
show $\fs_x(U)=U$ for just one points $x$. 

Let $\partial_L(U)$ be the projection of boundary of $\partial_L(\Sigma)$
to $\T(U)$. Consider a geodesic $g_U$ in $\T(U)$ connecting $a$ to $a'$ 
so that
\[
d_{\T(U)}\big( a, \partial_L(U)\big) = d_{\T(U)}\big(a', \partial_L(U)\big) = R,
\]
\[
d_{\T(U)}\big(g, \partial_L(U)\big) \geq R
\qquad\text{and}\qquad
d_{\T(U)} (a, a' ) \geq 4 \mul R.
\]
For example, we can choose a geodesic connecting two $L$--horoballs 
in $\T(U)$ that otherwise stays in the thick part of $\T(U)$ and then
we can cut off a subsegment of length $R$ from each end. 

Let $W = \Sigma -U$. Choose $b \in \T(W)$ to have distance $R$ from 
$\partial_L(W)$ and let $g$ be a path in $\T(\Sigma, L)$ that has constant projection 
to $W$ and projects to $g_U$ in $U$. That is,  
\[
g(t)_W=b
\qquad\text{and}\qquad
g(t)_U = g_U(t). 
\]
Then $g$ connects a point $x \in \T(\Sigma, L)$ to a point $x' \in \T(\Sigma,L)$ 
where $x_U=a$, $x'_U=a'$ and $x_W = x'_W = b$. 

Let, $y=f(x)$, $y'=f(x')$ and let $z$ and $z'$ be points on $\partial_L(\Sigma)$ 
that are distance $R$ to $x$ and $x'$ respectively. Since $f$ is anchored, 
we have 
\[
d_{\T} ( f(z), z) \leq \add
\qquad{and}\qquad 
d_{\T} ( f(z'), z') \leq \add,
\]
and hence 
\[
d_{\T}(y,z) \leq d_\T \big(f(x), f(z) \big) + d_\T \big(f(z), z \big)
\leq (\mul R + \add) + \add.  
\]
And the same holds for $d_{\T}(y',z')$. Also 
\begin{align*}
d_{\T(U)} (z_U, z'_U) 
 &\geq d_{\T(U)} ( x_U, x'_U) - d_{\T} ( z, x) - d_{\T} ( z', x') \\
  &\geq 4 \mul R- 2 R.
\end{align*} 
Therefore, 
\begin{align}
d_{\T(U)} ( y_U, y'_U) 
  & \geq d_{\T(U)} ( z_U, z'_U) - d_{\T(U)} ( y_U, z_U) - d_{\T(U)} ( y'_U, z'_U)\notag \\
  & \geq 4 \mul R - 2 R - 2(\mul R + 2\add) \geq  \mul R. \label{Eq:U-Progress} 
\end{align}

Now, choose points
\[
x= x_0, \ldots, x_N=x'
\]
along $g$ so that $d_\T(x_i, x_{i+1}) \leq r$ and 
$N = \frac{4 \mul R}{r}= 4  \frac{\mul}{\rho_1}$. 
As mentioned before, we already know $\fs_x(U) = \fs_{x_i}(U)$. If $\fs_x(U) \ne U$, 
then $\fs_{x_i}(U) \ne U$ and hence, by \propref{Prop:Local-Factors},
\[
d_{\T(U)}\Big( f(x_i)_{U},f(x_{i+1})_{U} \Big)\leq 2\Distance_1.
\]
Therefore, 
\[
d_{\T(U)}\Big( y_U, y_U' \Big)\leq 2N\Distance_1 \leq 8 \frac{\mul \Distance_1}{\rho_1}. 
\]
For $R$ large enough, this contradicts \eqnref{Eq:U-Progress}. 
Hence, $\fs_x(U) = U$. 
\end{proof}

\section{Nearly Shortest Curves} \label{Sec:Short-Curves}

The goal of this section is to prove \propref{Prop:Short-Curve}. Essentially, 
this states that if $\alpha$ is one of the shortest curves in $x$, then 
$\alpha$ is also short in $f(x)$ and, furthermore, $\fs_x(\alpha) = \alpha$. 

Throughout this section we always assume, if $\T = \T(\Sigma, L)$, that
$\Sigma$ does not have any components $U$ with $\chi(U)=1$.  
Hence,  \propref{Prop:Short-Curve} is the complementary statement to 
\corref{Cor:xi=1}. The effect of this assumption is that every curve has an 
adjacent curve. The arguments are conceptually very elementary. 
However, we need to keep careful track of constants. 

\subsection*{Cone is preserved}
For  $\eta>1$, define the $\eta$--maximal cone $\MC(\eta)$ to be the set of 
points $x \in \T$ so that 
\begin{itemize}
\item $\calS_x$ is a pants decomposition.
\item for $\alpha, \beta \in \calS_x= P_x$, 
\[
\frac{\tau_x(\alpha)}{\tau_x(\beta)} \leq \eta. 
\]
\end{itemize}
Let $\displaystyle \tau_x = \max_{\gamma \in P_x} \tau_x(\gamma)$. 
For point $z$ to be in $\LR$, $z$ has to contain two adjacent curves
that have lengths larger than $\ell_0$. Hence, for $x \in \MC(\eta)$, we have 
\begin{equation} \label{Eq:Distance-LR} 
\frac{\tau_x}{\eta} \leq d_\T(x, \LR) \leq \tau_x.
\end{equation}

\begin{proposition} \label{Prop:Max-Preserved}
For every $\eta_0$, there  is a $\tau_0$ so that  if $x\in \MC(\eta_0)$ and 
$\tau_x\geq \tau_0$ then
\[
f(x)\in \MC \big( 16 \mul^2\eta_0\big).
\]
In fact, for every $\gamma \in P_x$, we have
\[
\frac{\tau_x}{4\mul \eta_0} \leq \tau_{f(x)}(\gamma) \leq 4 \mul \tau_x.
\]
\end{proposition}

\begin{proof}  
Choose $R \geq \distance_1$ large enough so that if 
$d(x,\LR)\geq R$ then, similar to \eqnref{Eq:2K}, we have 
\begin{equation} \label{Eq:Distance-f-LR}
\frac {d(x,\LR)}{2\mul} \leq d(f(x),\LR)\leq 2\mul d(x,\LR),
\end{equation} 
and so that 
\begin{equation} \label{tau}
r=\rho_1R \geq 32\mul \, \eta_0 \, \Distance_1.
\end{equation}  
Let $\tau_0=16 \mul^2 \eta_0 R$. 

\subsection*{Claim} For adjacent curves $\beta, \beta' \in P_{f(X)}$ we have
\[
\tau_{f(x)} (\beta) \leq  \frac{\tau_x}{4\mul\eta_0}
\qquad \Longrightarrow \qquad
\tau_{f(x)}(\beta')\geq 4\mul \tau_x.
\]

\begin{proof} [Proof of Claim] \renewcommand{\qedsymbol}{$\blacksquare$}
We will later prove that in fact $\tau_{f(x)}(\beta')\leq 4 \mul \tau_x$, finishing
the proof of the Proposition. 

We prove the claim by contradiction. Assume that the first inequality in the claim
holds and the second is false. Note that we still have, by 
Equations \eqref{Eq:Distance-LR} and \eqref{Eq:Distance-f-LR}, that 
\[
 \max \left( \tau_{f(x)}(\beta'), \tau_{f(x)}(\beta) \right) 
 \geq d_\T(f(x) , \LR)  \geq \frac{ d_\T(x, \LR)}{2 \mul} \geq \frac{\tau_x}{2\eta_0 \mul}. 
\]
To summarize, we have two adjacent curves $\beta$ and $\beta'$ with
\[
\tau_{f(x)}(\beta) \leq \frac{\tau_x}{4\mul\eta_0}
\qquad\text{and}\qquad
 \frac{\tau_x}{2\mul\eta_0} \le \tau_{f(x)}(\beta') \leq 4 \mul \tau_x.
\]

Consider a geodesic $g$ moving only in the $\beta'$ factor that increases the
length of $\beta'$ and connects $f(x)$ to a point $f(x')$ with 
\[
\tau_{f(x')}(\beta') = \frac{\tau_x}{4\mul\eta_0}\geq \frac{\tau_0}{4\mul\eta_0} \geq R.
\]
We have \begin{equation}
\label{eq:distanceLR}
d\big( f(x'),\LR \big)\leq \frac{\tau_x}{4\mul\eta_0}.
\end{equation}
Take a sequence of points 
\[
f(x)=y_0, y_1, \ldots, y_N=f(x')
\]
along $g$ with $d_\T(y_i, y_{i+1}) \leq r$ and
\[
N=\frac{4 \mul \tau_x}{r}.
\]
Let $h = f^{-1}$, let $x^i = h(y_i)$, let $\calV_i$ be the $R$--decomposition at 
$y_i$ and $\calU_i$ be the $\frac{R}{2\mul}$--decomposition at $x_i$. 
By assumption, $\tau_{y_i}(\beta') \geq R$ and hence $\beta' \in \calV_i$ for every $i$
and, since the length of no other curve is changing along $g$, all $\calV_i$ are 
in fact the same decomposition (which we denote by $\calV$). Also $\calU_1 = P_x$. 
Let $\alpha' = h^\star_y(\beta')$, 

Let $\alpha\in P_x$ with $\alpha\neq \alpha'$.  Let $V \in \calV$ 
be the component so that $h^\star_{y_1}(V) = \alpha$, We show by induction on $i$, that  
$\alpha \in \calU_i$ and $h^\star_{y_i}(V)=\alpha$.   Assume this for $1 \leq i < j$.  Since 
$y_{i+1} \in B_r(y_i)$, and $y_i$ and $y_{i+1}$ have the same projection to 
$\T(V)$, \propref{Prop:Local-Factors} implies
 \[  
 d_{\T(\alpha)}(x^i_\alpha, x^{i+1}_\alpha )  \leq 2\Distance_1. 
 \]
Therefore, 
\[
d_{\T(\alpha)}(x^1_\alpha, x^j_\alpha) \leq 2 j \Distance_1 
  \leq 2 N \Distance_1 
  \leq  \frac{8 \mul \tau_x}r \Distance_1 
  \leq \frac{\tau_x}{4\eta_0}. 
\]
The last inequality is from the assumption on $r$. Hence, 
\[
\tau_{x^j}(\alpha) \geq \frac{\tau_x}{\eta_0} - \frac{\tau_x}{4 \eta_0}
\geq \frac{3\tau_x}{4\eta_0}. 
\]
This means in particular  that $\alpha$ is short enough 
($\tau_{x^j}(\alpha) \geq \frac{R}{2\mul}$)
and is included in $\calU_j$. Moreover by \propref{Prop:Analytic-Continuation},  
$h_{y_j}^\star(V)=\alpha$, completing the induction step.  Continuing this way, we 
conclude that, for every $\alpha \in P_x$, $\alpha \ne \alpha'$, we have 
\[
\tau_{x'}(\alpha) \geq \frac{3\tau_x}{4 \eta_0}. 
\]
In particular, for every pair of adjacent curves in $x'$, one satisfies the
above inequality. But
\[
d_\T(x', \LR) \leq 2 \mul d_\T(f(x'), \LR) \leq \frac{\tau_x}{2\eta_0}.
\]
This is a contradiction, which prove the claim. 
\end{proof}

We now show that, all curves $\beta \in P_{f(x)}$ must in fact satisfy 
\[
\tau_{f(x)}(\beta)\leq 4\mul \tau_x.
\]
The argument is similar to the one above, so we skip some of the details. 
Suppose this is false for some curve $\beta$. Let $\alpha$ be a curve so that
$f_x^*(\alpha) = \beta$ and $\alpha'$ be a curve adjacent to $\alpha$. 
Let $g$ be a geodesic that moves only in the $\alpha$ and $\alpha'$ factors,
increasing their lengths,  that connects $x$ to a point $x'$ where 
\[
\tau_x(\alpha) = \tau_x(\alpha') = R. 
\]
We cover $g$ with points 
\[
x=x_1, \ldots, x_N= x'
\]
so that $d(x,x') \leq r$ and $N= \frac{\tau_x}r$. Let $\calU_i$ be the 
$R$--decompositions at $x_i$ and $\calV_i$ be the $\frac{R}{2\mul}$--decomposition 
at $y_i = f(x_i)$. Then, as before, $\alpha$ and $\alpha'$ are in every $\calU_i$ 
(in fact, all $\calU_i$ are the same decompositions which we denote by $\calU$). 
Let $\fs_x(\alpha')=V'$. By an argument as above, the total movement in any other 
factor (besides $\beta$ and $V'$) is at most $2 N \Distance_1$. 

Since $d_\T(x_N, \LR) = R$, we have $d_\T(y_N, \LR) \leq 2 \mul R$. 
That is, there are two adjacent curves $\gamma$ and $\gamma'$ with 
\begin{equation} \label{Eq:Two-gammas}
\tau_{y_N}(\gamma) , \tau_{y_N}(\gamma') \leq 2\mul R. 
\end{equation}
But we have
\[
\tau_{y_N}(\beta) \geq \tau_{y_1}(\beta) - N (\mul r + \add) 
\geq 4 \mul \tau_x -  \frac{\tau_x}{r} (\mul r-\add) \geq  \frac{5}{2}\mul \tau_x
\] for $R$ large enough.
Hence, neither $\gamma$ or $\gamma'$ equals $\beta$. And one,
say $\gamma'$ is not contained in $V'$. Therefore, 
\begin{align*}
\tau_{y_1}(\gamma') 
& \leq \tau_{y_N}(\gamma') + 2 N \Distance_1 
    \leq  2\mul R + 2 \frac{\tau_x}r \Distance_1 \\
  &\leq \frac{2 \mul R}{\tau_0} \tau_x + \frac{2\Distance_1}{r}\tau_x
     \leq \left( \frac{1}{8\mul \eta_0} +  \frac{1}{16 \mul \eta_0} \right)\tau_x
     < \frac{\tau_x}{4 \mul \eta_0}.
\end{align*}
Now, since $\gamma$ is adjacent to $\gamma'$, the claim implies, 
$\tau_{y_1}(\gamma) \geq 4 \mul \tau_x$. However, again as above, 
\[
\tau_{y_N}(\gamma') \geq \tau_{y_1}(\gamma') - N (\mul r+\add)
\geq 4 \mul  \tau_x -  \frac{\tau_x}r (\mul r+\add) \geq \frac{5}{2} \mul \tau_x.
\]
This contradicts \eqnref{Eq:Two-gammas}. We are done.  
\end{proof}

\subsection*{Induced map on the curve complex is cellular.}
Let $\calP(S)$ be the complex of pants decompositions of $S$. 
For a pants decomposition $P \in \calP(S)$, define $\MC(P, \eta_0)$ 
to be the set of points $x \in \MC(\eta_0)$ where $P_x = P$. 
Note that $\MC(\eta_0)$ is not connected and its connected components
are parametrized by $\calP(S)$: 
\[
\MC(\eta_0) =  \coprod_{P \in \calP(S)} \MC(P, \eta_0). 
\]
The same is true for  $\MC(16\mul^2\eta_0)$ and, 
by \propref{Prop:Max-Preserved},
\[
f\big( \MC(\eta_0) \big)  \subset \MC(16 \mul^2\eta_0).
\] 
Hence, we can define a bijection $\fsP \from \calP(S) \to \calP(S)$ so that 
\[
f\big(\MC(P, \eta_0) \big) 
   \subset \MC\big( \fsP(P), 16 \mul^2\eta_0 \big).
\]
We will show that $f^\star_\calP$ is induced by a simplicial automorphism
 $\fsC$ of the curve complex. Note that the definition of $f^\star_\calP$
 depends on the choice of $\eta_0$. 

\begin{proposition}
\label{Prop:Ivanov}
Assuming $\eta_0$ large enough, there is a simplicial automorphism 
$\fsC \from \calC(S) \to \calC(S)$ apriori depending on $\eta_0$ so that, for a pants decomposition 
$P= \{ \alpha_1, \ldots, \alpha_n\}$ we have
\[
\fsP (P) = \big\{ \fsC(\alpha_1), \ldots, \fsC (\alpha_n) \big\}.
\]
\end{proposition}

\begin{proof}
Choose $\eta_0$ and $d$ so that
\[
\eta_0 d \geq R \geq \distance_1
\qquad\text{and}\qquad
r= \rho_1 R \geq 2d. 
\]
Let $\balpha$ be a multi-curve containing $\xi(S) -1$  curves and let 
\[
P = \balpha \cup \{ \beta\}
\qquad\text{and}\qquad
P' = \balpha \cup \{ \beta'\}
\] 
be two extensions of $\balpha$ to pants decompositions with
$\I(\beta, \beta') \leq 2$. Let $x$ be a point in $\MC(P, \eta_0)$ so that 
$\tau_x(\beta) = d$ and $\tau_x(\alpha) = \eta_0 d$ for $\alpha \in \balpha$.
Similarly, define $x' \in \MC(P', \eta_0)$ so that 
\[
\forall \alpha \in \balpha \quad x_\alpha' = x_\alpha, \qquad
\tau_{x'}(\beta') = d
\qquad\text{and}\qquad 
d_{\T}(x, x') \leq 2d.
\]
Let $\calU$ be the decomposition consisting of $\balpha$ and the complementary 
subsurface $U$. Then $x, x' \in T_\calU$. Let $g$ be a path connecting 
$x$ to $x'$ whose projection is constant in every $\T(\alpha)$ for $\alpha \in \balpha$, 
is a geodesic connecting $x_U$ to $x'_U$ in $\T(U)$ 
and has a length $2d$. We have 
\[
d_\T(g, \LR) \geq \eta_0 d \geq R.
\] 
Hence \propref{Prop:Local-Factors} and \propref{Prop:Analytic-Continuation}
apply to $x$ and $x'$. This means $f(x)$ and $f(x')$ are contained in 
the same product region, say $T_\calV$. Let $V = \fs_x(U)$. Then, for
$W \in \calV$, $W \ne V$, 
\[
d_{\T(W)}(x_W, x'_W) \leq \Distance_1. 
\]
But 
\[
f(x) \in \MC\big( \fsP (P), 16\mul^2\eta_0\big)
\quad\text{and}\quad
 f(x') \in \MC \big(\fsP (P'), 16\mul^2\eta_0\big).
\]
Hence, $\calV$ contains a multi curve $\bgamma$ with  $(\xi(S)-1)$ curves 
and $\fsP (P)$ and $\fsP (P')$ share all but one curve. That is, there are curve 
$\delta$ and $\delta'$ so that 
\[
\fsP (P)= \bgamma \cup \{\delta\}
\quad\text{and}\quad
\fsP (P')= \bgamma \cup \{\delta'\}.
\]
However, for the moment, we do not have a good bound on the intersection 
number $\I(\delta, \delta')$. 

We first show that the multi curve $\bgamma$ does not depend on 
the choice of $\beta'$. Assume $\beta''$
is another curve with $i(\beta, \beta'') \leq 2$, and let $P'' = \balpha \cup \{\beta''\}$. 
If $\I(\beta,' \beta'') \leq 2$ then $\fsP(P)$, $\fsP(P')$ and $\fsP(P'')$ share
$(\xi(S)-1)$ curves. Hence these have to be the same multicuve $\bgamma$. 
If $\I(\beta,' \beta'')$ is large, then we can find a sequence
\[
\beta' = \beta_1 \ldots, \beta_m= \beta''
\] 
of curves that are disjoint from $\bgamma$, $\I(\beta, \beta_i) \leq 2$ and 
$\I(\beta_i, \beta_{i+1}) \leq 2$.
Define $P_i =\balpha \cup \{\beta_i\}$. Then arguing as above shows that 
all $\fsP(P_i)$ share the same $(\xi(S)-1)$ curves. That is, the same
curve changes from $\fsP(P)$ to $\fsP(P'')$ as it did from $\fsP(P)$ 
to $\fsP(P')$ and $\bgamma \subset \fsP(P'')$. 

Note that  we have shown that there is an association between curves in $P$ and $\fsP (P)$; a curve 
$\beta \in P$ is associated to the curve $P(\beta) \in \fsP(P)$ that is not contained in 
any $\fsP (P')$ constructed as above. We will show $P(\beta)$ is the same for every $P$.

Let $\balpha$ be a multi-curve with $(\xi(S)-2)$ curves and let
$P = \balpha \cup \{\beta_1, \beta_2\}$ be a pants decomposition. 
Let 
\[
P_1 = \balpha \cup \{\beta_1', \beta_2\}, \quad
P_2= \balpha \cup \{\beta_1, \beta_2'\}, \quad\text{and}\quad
P_{12}= \balpha \cup \{\beta_1', \beta_2'\},
\] 
with $\I(\beta_i, \beta_i') \leq 2$ for $i=1,2$. Denote 
\[
Q= \fsP(P), \quad Q_1= \fsP(P_1), \quad Q_2= \fsP(P_2)
\quad\text{and}\quad Q_{12}= \fsP(P_{12}).
\] 
Let $\bgamma$ be a multi-curve with $(\xi(S) -2)$ curves so that 
$Q= \bgamma \cup \{P(\beta_1), P(\beta_2)\}$. From the discussion above, 
we know that, there are curves  $\delta_1, \delta_2$ so that.
\[
Q_1 = \bgamma \cup \{\delta_1, P(\beta_2) \}, 
\qquad\text{and}\qquad
Q_2= \bgamma \cup \{P(\beta_1), \delta_2\}.
\]
Since $P(\beta_1)$ and $P(\beta_2)$ are disjoint, the curves 
$\delta_1$ and $\delta_2$ must be different. Also $Q_{12}$ shares $(\xi(S)-1)$ 
curves with both $Q_1$ and $Q_2$. But the map $\fsP$ is a bijection and 
$Q_{12} \not = Q$. Moreover since $\delta_1\neq\delta_2$ and $\beta_1$ and $\beta_2$ are disjoint $Q_{12}$ cannot contain both $P(\beta_1)$ and $\delta_1$ and similarly it cannot contain both $P(\beta_2)$ and $\delta_2$.  Therefore, 
\[
Q_{12}= \bgamma \cup \{\delta_1, \delta_2\}.
\]
That means, $P_2(\beta_1) = P(\beta_1)$ because it is the curve
that changes from $Q_2$ to $Q_{1,2}$. Similarly, 
$P_1(\beta_2) = P(\beta_2)$.

W have shown the association $\beta \to P(\beta)$ is the same for 
adjacent pants decompositions in $\calP(S)$. Hence, it does not depend on $P$ and 
we can define $\fsC(\beta) = P(\beta)$ for any pants decomposition $P$ containing 
$\beta$. 

Every multi-curve is contained in a pants decomposition. Hence
$\fsC$ sends disjoint curves to disjoint curves and in fact sends simplifies in $\calC(S)$ 
to simplices. Since $\fsP$ is onto, $\fsC$ is also onto. We show $\fsC$ is one-to-one.
Assume for contradiction that $\fsC(\alpha) = \fsC(\beta)$ and let
$P_\alpha$ and $P_\beta$ be pants decompositions that contain $\alpha$
and $\beta$ respectively but have no curves in common. 
Then $\fsP(P_\alpha)$ and $\fsP(P_\beta)$ share a curve $\delta$. Let 
\[
\fsP(P_\alpha)= Q_1, \ldots, Q_m =\fsP(P_\beta)
\]
be a sequence of adjacent pants decomposition all containing $\delta$
and let $P_i = \fsP^{-1}(Q_i)$. Then $\fsC^{-1}(\delta)$ is contained in 
every $P_i$. This contradicts the assumption that $P_\alpha$ and
$P_\beta$ have no common curves. 

We have shown $\fsC$ is simplicial and it is a bijection, that is, it is
a simplicial automorphism of $\calC(S)$. 
\end{proof}

\begin{remark} \label{Rem:Identity} 
In the case $\T = \T(S)$, by a theorem of Ivanov \cite{ivanov:ACC}, $\fsC$ is induced 
by an isometry $\fs$ of \Teich space $\T(S)$. Hence, after applying the inverse of this 
isometry to $f$, we can assume that $\fsC$ is the identity map. In the case
$\T= \T(\Sigma, L)$ and $f$ is anchored, we know that $\fsP$ is the identity map. 
Thus, so is $\fsC$. Indeed, for the remainder of the paper, we assume that 
in these cases $\fsC$ is the identity. 
\end{remark}

\begin{remark}  We point out that $\fsC$ and $f_x^\star$ are different maps.  We now show that on each $\eta_0$ cone the latter map is the identity as well. 
\end{remark}

\begin{corollary} \label{Cor:f-star}
Assume either $f=f_S$ or $f=f_\Sigma$ is anchored. Let $x \in \MC(P, \eta_0)$ 
be a point with 
\[
\tau_x \ge \max (\tau_0, 2\eta_0 \distance_1) .
\] 
Then for $\alpha \in P$, we have $\fs_x(\alpha) = \alpha$.
\end{corollary} 

\begin{proof}
Let $R= \frac{\tau_x}{2\eta_0}$. Then the $R$--decomposition at $x$ is $P_x$
and $\frac{R}{2\mul}$--decomposition of $f(x)$ is also $P_x$. This is because, 
by \remref{Rem:Identity} and \propref{Prop:Max-Preserved} we have, 
for $\alpha \in P_x$
\[
\tau_{f(x)}(\alpha) \geq \frac{\tau_x}{4 \mul \eta_0} = \frac{R}{2 \mul}. 
\]
Hence $\fs_x(\alpha)$ is some curve in $P_x$. 

Take $\gamma \in P$, $\gamma \ne \alpha$, let $\gamma'$ be a curve interesting 
$\gamma$ once or twice that is disjoint from other curves in $P$ and let 
\[
P' = (P - \{\gamma\}) \cup \{ \gamma'\}. 
\]
Let $x' \in \MC(P', \eta_0)$ be a point so that
\[
\forall \beta \in P- \{\gamma\} \quad x_\beta = x'_\beta
\qquad\text{and}\qquad
\tau_x(\gamma) = \tau_{x'}(\gamma').
\]
Let $g$ be a path connecting $x$ to $x'$ that is constant in all other factors. 
As  argued before, since $\alpha$ remains short along $g$, the length of 
 $\fs_x(\alpha)$ changes by at most $2 N \Distance_1$ where 
 $N = \frac{\tau_x}{\rho_1R}$. This is less than the change in the  length  of $\gamma$. 
 Hence $\fs_x(\alpha) \ne \gamma$. 
Since we can do this argument for every curve $\gamma \ne \alpha$, 
we conclude $\fs_x(\alpha) = \alpha$. 
\end{proof} 

\subsection*{The restriction to the thick part.} 

\begin{proposition}
\label{Prop:Thick}
Assume $\fsP$ is the identity. Then, there is constant $\Dt$ so that, if $x$ 
is $\ell_0$--thick, then
\[
d_\T\big(f(x), x)\big) \leq  \Dt. 
\]
\end{proposition}

\begin{proof}
Let $\sf B>0$ be the Bers constant which has the property that  the set of curves 
of length at most $\sf B$ fill $x$. To find an upper-bound for $d\big( x,f(x)\big)$ 
it is enough to show that, for each $\alpha$ with $\Ext_x(\alpha) \leq  \sf B$, $\alpha$ 
has  bounded length in $f(x)$. (This follows, for example, from \cite[Theorem B]{rafi:LT} 
and the fact that extremal length and hyperbolic lengths are comparable in a thick
surface $x$.) 

Let $\eta_0$ be as in \propref{Prop:Ivanov} and $\tau_0$ be as in 
\propref{Prop:Max-Preserved}. For $\alpha$ as above, let $P$ be
a pants decomposition containing $\alpha$ and $z \in \MC(P, \eta_0)$ 
be such that $\tau_y = \tau_0$ and $d_\T(x, z) \ladd \tau_0$. 
Then, by \propref{Prop:Max-Preserved}, 
\[
\tau_{f(z)}(\alpha) \geq \frac{\tau_0}{4 \mul \eta_0}.
\]
Now we have
\[
\mul \tau_0 \gadd d_\T(f(x), f(z)) 
\gadd \log\frac{\Ext_{f(x)}(\alpha)}{\Ext_{f(z)}(\alpha)}
\gadd \log \Ext_{f(x)}(\alpha). 
\]
Hence
\[
 \log \Ext_{f(x)}(\alpha) \ladd \mul \tau_0. 
\]
That is, the length of all such $\alpha$ in $f(x)$ is uniformly bounded
and hence $d_\T(x, f(x))$ is uniformly bounded as well. 
\end{proof}

\subsection*{Grouping of sizes}
\begin{definition} \label{Def:Admissible-Scale}
Fix once and for all 
\[
\mu= 64\mul^2\Distance_1^2. 
\]
Suppose we have a set $\{d_1, d_2, \ldots, d_m\}$ of positive numbers. We say 
\emph{$\eta$ is an admissible scale for this set} if there is a decomposition 
$\calE$ of this set so that, if $d_i, d_j \in E$ for $E \in \calE$ then 
$\frac{d_j}{d_i}\leq \eta$, and if  they are in different subsets   
then $\frac{d_j}{d_i}\geq\mu\,\eta$. We refer to $\calE$ as the 
\emph{the partition associated to $\eta$}. We call the set $E \in \calE$
containing the largest elements as the \emph{top group}. 

\end{definition}

\begin{lemma}
\label{Lem:Group}
Given $\eta_0>1$, there are scales $\eta_0<\eta_1<\ldots <\eta_m$ such that for 
any set $\{d_1, d_2, \ldots, d_m\}$ of positive numbers, some $\eta_i$ is an 
admissible scale for this set. In fact, we can define $\eta_i$ recursively as
\[
\eta_{i+1} = \mu\, \eta_i^3. 
\]
\end{lemma}

\begin{proof}
For $1\leq i \leq m$, let $\calE_i$ be the  partition of $\{d_1, d_2, \ldots, d_m\}$ containing the fewest number of subsets 
so that, for $E \in \calE_i$ and $d,d' \in E$, we have $\frac{d}{d'} \leq \eta_i$. 
If we also have  $\frac{d}{d'} \geq \mu \eta_i$ for $d,d'$ is different sets
then $\eta_i$ is an admissible scale and we are done. Otherwise, 
for every $i$, there are two sets $E, E' \in \calE_i$, $d \in E$ and $d' \in E'$ so that 
$\frac{d}{d'} \leq \eta \mu$. Then, for any $c \in E$ and $c' \in E'$ 
we have 
\[
\frac{c}{c'} 
\leq \frac{c}{d} \cdot  \frac{d}{d'} \cdot  \frac{d'}{c'} 
\leq \mu \eta_i^3 = \eta_{i+1}. 
\]
Which means, $E$ and $E'$ fit in one group in $\calE_{i+1}$ and 
\[
\Big| \calE_{i+1} \Big| \leq \Big| \calE_i \Big| - 1. 
\]
If this holds for all $i$, then some $\calE_j$ has size one and $\eta_j$ would
be an admissible scale.   
\end{proof}

\subsection*{The shortest curves are preserved}

\begin{proposition} \label{Prop:Short-Curve}
Assume either $f=f_S$ of $f=f_\Sigma$ is anchored.
For $\tau_1$ large enough, the following holds. Let $x \in \T$ be a point 
such that, for $\alpha \in P_x$, $\tau_x(\alpha) \geq \tau_1$ and let $\eta$ be
an admissible scale for the set $\big\{ \tau_x(\alpha) \big\}_{\alpha \in P_x}$ with
the associated partition $\calE$. Let $E$ be the top group in $\calE$ and
let $\balpha$ be the set of curves $\alpha$ where $\tau_x(\alpha) \in E$. 
Then, for every $\alpha \in \balpha$, we have
\[
\frac{\tau_{x}}{\sqrt \mu \eta} 
\leq  \tau_{f(x)}(\alpha)
\leq 2\mul \tau_{x},
\]
and $\fs_x(\alpha) = \alpha$. 
\end{proposition}

\begin{proof}
For $\tau_1$ large enough,
\begin{align*}
 \tau_{f(x)}(\alpha) & \leq d_\T(f(x), \Tt) \leq \mul d_\T(x, \Tt) + \add + \Dt\\
   & \leq \mul \tau_x + \add + \Dt \leq 2 \mul \tau_x.
\end{align*}
Hence, we have the upper-bound. We also require that 
\[
\tau_1>32\Distance_1\mul \eta_m/\rho_1,
\] 
where $\eta_m \geq \eta$ is from \lemref{Lem:Group}. Let $r=\rho_1\tau_1$ and 
let $y = f(x)$.  Let $\bbeta$ be the set of curves $\beta$ with
$\tau_{y}(\beta) \geq \frac{\tau_x}{\sqrt\mu\eta}$.

\subsection*{Claim} We have
\[
|\bbeta| \leq |\balpha|. 
\]

\begin{proof}[Proof of Claim] \renewcommand{\qedsymbol}{$\blacksquare$}
The proof is essentially the same at the proof of \propref{Prop:Max-Preserved}. 
Choose a path $g$ that changes the length of curves in $\balpha$ only
connecting $x$ to a point $x'$ so that, for $\alpha \in \balpha$, 
$\tau_{x'}(\alpha) = \frac{\tau_x}{\eta \mu}$ and $d_\T(x, x') \leq \tau_x$. 
Since all other curves in $x$ are already shorter that $\frac{\tau_x}{\eta \mu}$,
we have 
\[
\tau_{x'} = \frac{\tau_x}{\eta \mu}.
\] 
We can cover $g$ with points 
\[
x= x_1, \ldots, x_N = x',
\]
so that $d_\T(x_i, x_{i+1}) \leq r $ and $N = \frac{\tau_x}{r}$. Let $k = | \balpha |$
and $y' = f(x')$. Then, as in the proof of \propref{Prop:Max-Preserved}, 
only the length of  $k$ curves can changes substantial from $y$ to $y'$. More
precisely, if $\bbeta$ has more than $k$ curves, then there is $\beta \in \bbeta$ 
so that 
\[
\big| \tau_y(\beta) - \tau_{y'}(\beta) \big|  \leq 2 N \Distance_1. 
\]
Also, $\tau_{y'} \leq 2 \mul \tau_{x'} = \frac{2 \mul \tau_x}{\eta \mu}$. Hence, 
\[
\tau_y(\beta) \leq \tau_{y'}(\beta) + 2 N \Distance_1 
\leq \frac{2 \mul \tau_x}{\eta \mu} + \frac{2 \Distance_1 \tau_x}r 
\leq \frac{\tau_x}{3\sqrt \eta \mu}.
\]
This  is a contradiction to the assumption $\beta \in \bbeta$. This proves the claim. 
\end{proof}

We now show that, in fact, $\balpha = \bbeta$. Assume there is a curve 
$\alpha \in \balpha - \bbeta$. Let $w$ be a point so that
\[
\forall \beta \in \bbeta \quad w_\beta = y_\beta,  \qquad
\forall \beta \in P_y - \bbeta \quad \tau_w(\beta) = 1
\]
and
\[
d_\T(y,w) \leq \frac{\tau_x}{\sqrt \mu \eta}. 
\]
Let $P$ be a pants decomposition containing $\bbeta$ so that
\[
\I(P, \alpha) \ne 0, \qquad\text{and}\qquad
\forall \beta' \in P - \bbeta \quad \Ext_w(\beta') \leq \Bers. 
\]
Let $y' \in \MC(P, \sqrt \mu \eta)$ be a point obtained from $w$ by pinching curves 
$\beta' \in  P - \bbeta$ until $\tau_{y'}(\beta') = \tau_y/\sqrt \mu \eta$. Then 
\[
d_\T(w, y') \leq  \frac{\tau_y}{\sqrt \mu \eta} \leq \frac{2\mul \tau_x}{\sqrt \mu \eta}
\qquad\text{and hence}\qquad
d_\T(y, y') \leq \frac{(2 \mul +1) }{\sqrt \mu \eta} \tau_x .
\]

By \propref{Prop:Max-Preserved}, all curves in $P$ are still short in $x'=f^{-1}(y')$. 
This means, $\alpha$ is not short, and hence
\[
d_\T(x, x') \geq \tau_{x}(\alpha) \geq  \frac{\tau_x}{\eta}. 
\]
But we also have
\[
d_\T(x, x') \le \mul d_\T(y, y') + \add 
\leq \frac{(2 \mul^2 +\mul) }{\sqrt \mu \eta} \tau_x +\add.
\]
This is a contradiction which proves $\balpha \subset \bbeta$. This and the claim
imply $\balpha = \bbeta$. 

We now show $\fs_x(\alpha) = \alpha$ for $\alpha \in \balpha$. 
This essentially follows from \corref{Cor:f-star}. Let $g$ be a path connecting
$x$ to a point $x' \in \MC(P_x, \eta_0)$ that is constant in the $\alpha$
coordinate (that is, changes the length of all curves until they are comparable
to $\alpha$). Note that $\tau_{x'}$ is large enough that \corref{Cor:f-star} applies. 
But, by \corref{Cor:f-star} $\fs_{x'}(\alpha) = \alpha$. And as,
we have argued in the proof of claim in \propref{Prop:Max-Preserved},
as we move along $g$ from $x'$ to $x$, $\alpha$ remains short both in $g(t)$ and 
$f\big( g(t)\big)$.  Hence, \propref{Prop:Analytic-Continuation} applies to 
all points along this path and $\fs_x(\alpha) = \alpha$ as well. We are done. 
\end{proof}

\section{Applying induction} 
\label{Sec:Induction}
We start by proving the base case of induction
(\thmref{Thm:Induction}). Note that when $\xi(\Sigma) =1$, the surface $\Sigma$
is connected and is either a punctured torus or a four-times punctured sphere and
$\T(\Sigma) = \H$. 
\begin{proposition} \label{Prop:Base-Case}
Assume $\xi(\Sigma)=1$ and $\mul_\Sigma$ and $\add_\Sigma$ are given. 
Then \thmref{Thm:Induction} holds for $\Length_\Sigma=0$ and some constant 
$\Distance_\Sigma$. 
\end{proposition} 

\begin{figure}[ht]
\setlength{\unitlength}{0.01\linewidth}
\begin{picture}(100, 30)

\put(35,0){
  \begin{tikzpicture}
   [thick, 
    scale=.5\unitlength,
    vertex/.style={circle,draw,fill=black,thick,
                   inner sep=0pt,minimum size=1mm}]
                   
  \draw[fill=gray!40]  (0,0) circle (1);

  \node [draw, fill=white] at (0: .76) [circle through={(0: 1)}] {};
  \node [draw, fill=white] at (90: .76) [circle through={(90: 1)}] {};
  \node [draw, fill=white] at (180: .76) [circle through={(180: 1)}] {};
  \node [draw, fill=white] at (-90: .76) [circle through={(-90: 1)}] {};

  \node [draw, fill=white] at (45: .87) [circle through={(45: 1)}] {};
  \node [draw, fill=white] at (135: .87) [circle through={(135: 1)}] {};
  \node [draw, fill=white] at (-45: .87) [circle through={(-45: 1)}] {};
  \node [draw, fill=white] at (-135: .87) [circle through={(-135: 1)}] {};

  \node [draw, fill=white] at (29: .94) [circle through={(29: 1)}] {};
  \node [draw, fill=white] at (61: .94) [circle through={(61: 1)}] {};
  \node [draw, fill=white] at (119: .94) [circle through={(119: 1)}] {};
  \node [draw, fill=white] at (151: .94) [circle through={(151: 1)}] {};
  \node [draw, fill=white] at (-29: .94) [circle through={(-29: 1)}] {};
  \node [draw, fill=white] at (-61: .94) [circle through={(-61: 1)}] {};
  \node [draw, fill=white] at (-119: .94) [circle through={(-119: 1)}] {};
  \node [draw, fill=white] at (-151: .94) [circle through={(-151: 1)}] {};

  \node [draw, fill=white] at (70: .975) [circle through={(70: 1)}] {};
  \node [draw, fill=white] at (20: .975) [circle through={(20: 1)}] {};
  \node [draw, fill=white] at (110: .975) [circle through={(110: 1)}] {};  
  \node [draw, fill=white] at (160: .975) [circle through={(160: 1)}] {};
  \node [draw, fill=white] at (-70: .975) [circle through={(-70: 1)}] {};
  \node [draw, fill=white] at (-20: .975) [circle through={(-20: 1)}] {};
  \node [draw, fill=white] at (-110: .975) [circle through={(-110: 1)}] {};  
  \node [draw, fill=white] at (-160: .975) [circle through={(-160: 1)}] {};

  \draw (29: .88) -- (-151: .88); 
  \draw (0: .52) -- (180: .52); 
  \draw[red] (0: .2) arc (0:29: .2); 
  
  \node[vertex] (z) at (0,0) {}; 
        
  \end{tikzpicture}}

\put(29, 25){$\T(\Sigma, L)$} 
\put(49.5, 12){$z$} 
\put(44,9){$g_2$}
\put(44,16){$g_1$}
\put(55, 15){$\theta$} 

\end{picture}

\caption{The geodesics $g_1$ and $g_2$ pass through $z$, have their end points 
on $\partial_L(\Sigma)$ and the angle $\theta$ between them is of a definite size.}
\label{Fig:Base-Case}
\end{figure}
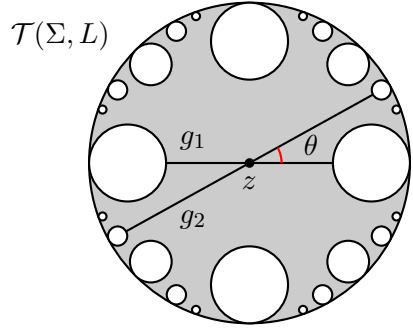

\begin{proof}
Let $L \geq \Length_\Sigma = 0$ be given. Consider a point $z \in \T(\Sigma, L)$. 
For $i=1,2$, let $g_i \from [a_i, b_i] \to \T(\Sigma, L)$ be a geodesic so that $g_i(0)=z$ 
and $g_i(a_i)$ and $g_i(b_i)$ lie on distinct $L$--horocycles and so that $g_1$ and $g_2$ 
have a definite angle between them (see \figref{Fig:Base-Case}). Define 
$\bg_i = f_\Sigma \circ g_i$. Then $\bg_i$ is a $(K_\Sigma, C_\Sigma)$--quasi-geodesic.   
Since $f_\Sigma$ is anchored,  the endpoints of $\bg_i$ are uniformly bounded distance from the endpoints of $g_i$.  This implies $\bg_i$  is contained in a uniform bounded neighborhood of $g_i$. This means,
$f(z)= \bg_i(0)$ is contained in a uniform bounded neighborhood of both $g_1$ and 
$g_2$. But, since $g_1$ and $g_2$ have a definite size angle between them, the 
diameter of this set of uniformly bounded. Hence $d_\H\big( z, f(z) \big)$ is uniformly 
bounded. 
\end{proof}

Our plan is to apply induction by removing from $\Sigma$ all components 
of complexity 1 and by cutting along the shortest curves. 

\begin{proposition} \label{Prop:Torus-Identity}
Assume $\Sigma$ has a component $W$ with $\xi(W) = 1$. 
Let $f \from \T(\Sigma, L) \to \T(\Sigma, L)$ be a 
$(\mul_\Sigma, \add_\Sigma)$--quasi-isometry that is anchored. 
Pick a large $R$ so that the statements in \secref{Sec:Local-Factors} apply.
Then, there is a constant $\Distance_W$ so that the following holds. 
Let $z \in \T=\T(\Sigma, L)$ be a point so that 
\begin{equation} \label{Eq:Far-Away} 
d_\T(z, \LR) \geq R
\qquad\text{and}\qquad
d_\T\big( z, \partial_L(\Sigma) \big) \geq R. 
\end{equation}
Then 
\[
d_{\T(W)}(z_W, f(z)_W) \leq \Distance_W. 
\]
\end{proposition}

\begin{proof}
Let $\Sigma'=\Sigma - W$. We denote a point $x \in \T(\Sigma)$ as a tuple
$(x_W, x_{\Sigma'})$. Let $r= \rho_1 R$. 

\subsection*{Claim 1} Let $g \from [a,b] \to \T(W)$ be a geodesic that
stays a distance at least $R$ from $\partial_L(W)$ and let
\[
x_t= \big( g(t), z_{\Sigma'}\big) \in \T(\Sigma, L).
\]
Define
\[
\bg(t) = f\big( x_t \big)_W. 
\]
Then $\bg$ is a quasi-geodesic.

\begin{proof}[Proof of Claim 1] \renewcommand{\qedsymbol}{$\blacksquare$}
To see the upper-bound, we note that, for times $s$ and $t$, 
\begin{align*}
d_{\T(W)} \big( \bg(t), \bg(s) \big) 
  & = d_{\T(W)} \Big( f(x_t )_{W}, f(x_s)_{W}\Big) 
   \leq d_{\T(\Sigma)} \big( f(x_t), f(x_s) \big) \\
  &\leq \mul \, d_{\T(\Sigma)} \big( x_t, x_s \big)+ \add 
  = \mul  \, |t-s| + \add.
\end{align*} 
We now check the lower bound. 
Pick a sequence of points in $\T(\Sigma, L)$
\[
x_s = x_1, \ldots, x_N = x_t
\]
so that $d_{\T(\Sigma)}(x_i,x_{i+1}) \leq r$ and $N \leq \frac{|t-s|}r+1$. 
Since $W$ is always a factor in any decomposition, we know from 
\propref{Prop:Local-Factors} that 
\[
d_{\T(\Sigma')} \Big(f(x_i)_{\Sigma'}, f(x_{i+1})_{\Sigma'} \Big) \leq 2 \Distance_1. 
\]
Therefore, (assuming $r \geq 4 \mul \Distance_1$) 
\[
 d_{\T(\Sigma')} \Big( f(x_s)_W, f(x_t)_W \Big) \leq 2 N \Distance_1
 \leq \frac{|t-s|}{2\mul} + 2\Distance_1.
\]
Now the desired lower bound in the claim follows:
\begin{align*}
d_{\T(W)} \big( \bg(s), \bg(t) \big) 
  & = d_{\T(W)} \big( f(x_s)_{W}, f(x_t)_{W}\big) \\
  & \geq d_{\T(\Sigma)} \big( f(x_s), f(x_t) \big) 
      - d_{\T(\Sigma')} \big( f(x_s)_{\Sigma'}, f(x_t)_{\Sigma'} \big)  \\
  &\geq \frac 1\mul \, d_{\T(\Sigma)} \big( f(x_s) , f(x_t) \big) 
      - \add  - \frac {|t-s|}{2\mul} - 2\Distance_1 \\ 
   &= \frac {|t-s|}{2\mul} - (\add+2\Distance_1). \qedhere
\end{align*} 
\end{proof}

Next, we show that, if the end points of $g$ are close to $\partial_L(W)$ then 
the end points of $\bg$ are close  to the end points of $g$ which would imply 
exactly as in the proof of Proposition~\ref{Prop:Base-Case} that $\bg$ stays near $g$. That is, the reader should think of $w$ below as
an end point of $g$. 

\subsection*{Claim 2} Let $w \in \T(W)$ be a point so that 
$d_{\T(W)}\big(w, \partial_L(W)\big) = R$. Then
\begin{equation} \label{Eq:End}
d_{\T(W)}\Big( f\big(w, z_{\Sigma'}\big)_W, w \Big) 
   \leq \frac{2L\Distance_1}{r} + (\mul +1) R + 2\add. 
\end{equation}

\begin{proof}[Proof of Claim 2] \renewcommand{\qedsymbol}{$\blacksquare$}
We choose a sequence of points in $\T(\Sigma')$
\[
z_{\Sigma'} = u_1, \ldots, u_N.
\]
where 
\[
d_{\T(\Sigma')}\big(\partial_L(\Sigma'), u_N \big)=R 
\qquad\text{and}\qquad
d_{\T(\Sigma')}\big(u_N , z_{\Sigma'}\big) \leq L. 
\]
Also 
\[
d_{\T(\Sigma')}(u_i, u_{i+1}) \leq r
\qquad\text{and}\qquad
N \leq \frac{L}r.
\] 
Let $z_i = (w, u_i)$. Note that, $z_1$ is the point of interest. 
By \propref{Prop:Local-Factors}, we have
\[
d_{\T(W)}\big( f(z_i)_W, f(z_{i+1})_W \big) \leq 2 \Distance_1. 
\]
Hence, 
\[
d_{\T(W)}\big( f(z_1)_W, f(z_N)_W \big) \leq 2 N \Distance_1 \leq \frac{2L\Distance_1}{r}. 
\]
But $z_N$ is distance $R$ from some point $z_L \in \partial_L(\Sigma)$. 
Hence
\begin{align*}
d_{\T(W)}\big( w, f(z_N)_W \big) 
  &\leq d_{\T}\big( z_N, f(z_N) \big)\\
  & \leq d_{\T}( z_N, z_L ) + d_{\T}\big( z_L, f(z_L) \big) + d_{\T}\big( f(z_L), f(z_N) \big)\\ 
   & \leq R + \add + \mul R + \add \leq (\mul + 1) R + 2\add.
\end{align*}
The claim follows from the last two inequalities by the triangle inequality. 
\end{proof}

We now prove the proposition. For $i=1,2$, consider the geodesic segment
$g_i\from [a_i,b_i] \to \T(W)$ where, $g_i(0) = z_W$, $g_i(a_i)$ and $g_i(b_i)$ are 
distance $R$ from $\partial_L(W)$ and so that $g_1$ and $g_2$ make a definite size 
angle at $z_W$. By Claim 1, the paths $\bg_i$ are quasi-geodesics and 
Claim 2, provides a bound for the distance between $g_i(a_i)$
and $\bg_i(a_i)$ and also between $g_i(b_i)$ and $\bg_i(b_i)$.
We can make $g_1$ and $g_2$ to be as long as needed. If the lengths of $g_i$ 
are long enough compared with the right hand side of \eqnref{Eq:End}, this implies 
that $\bg_i$ stays in a uniform neighborhood of $g_i$ and, in particular, 
$f(z)_W=\bg_i(0)$ is near $g_i$, for $i=1,2$. But $g_1$ and $g_2$ make a 
definite size angle. Hence $f(z)_W$ is near $z_W$. We are done. 
\end{proof}

\begin{proposition} \label{Prop:Shortest-Identity} 
Assume either $\T= \T(S)$ or $\T= \T(\Sigma, L)$ and $\Sigma$ has no component 
with  complexity one. Let $f \from \T \to \T$ be a $(\mul, \add)$--quasi-isometry 
so that the restriction of $f$ to the thick part is $\Dt$ close to the identity.  
Pick a large $R$ so that the statements in \secref{Sec:Local-Factors} apply.
Then there is a constant $\Distance_{\rm Top}$ so that the following holds. 
Let $z \in \T$ be a point so that (the second condition applies only when 
$\T=\T(\Sigma, L)$) 
\begin{equation*}
d_\T(z, \LR) \geq \mu \eta R
\qquad\text{and}\qquad
d_\T\big( z, \partial_L(\Sigma) \big) \geq  R,
\end{equation*}
and let $\alpha$ be the shortest curve in $z$; $\tau_z=\tau_z(\alpha)$. Then 
\[
d_{\T(\alpha)}\big( z_\alpha, f(z)_\alpha \big) \leq \Distance_{\rm Top}. 
\]
\end{proposition}

\begin{proof}
The proof is essentially the same as the proof of \propref{Prop:Torus-Identity}. 
Let $\Sigma' = S-\alpha$ or $\Sigma - \alpha$. For any geodesic 
$g \from [a,b] \to \T(\alpha)$ so that $g(0) = z_\alpha$ that stays $\tau_z/\eta_0$ away 
from $\LR$ and $\partial_L(\alpha)$, we let $x^t$ to be a point in $\T$ that projects to 
$g(t)$ in $\T(\alpha)$ and has the same projection to $\Sigma'$ as $z$. Then $\alpha$ 
is still in the top group of $x^t$. Define 
\[
\bg(t) = f(x^t)_\alpha. 
\]
Assuming $\Length_\Sigma$ is large enough, such geodesics exist. 
Also, by \propref{Prop:Short-Curve}, $\fs_{x^t}(\alpha)=\alpha$. 
It can be shown, as in claim 1 in \propref{Prop:Torus-Identity}, that 
$\bg$ is a quasi-geodesic with uniform constants. We choose the end points of $g$ 
so that $\tau_{x^a}(\alpha) = \tau_{x^b}(\alpha) = \tau_z/\eta_0$.

Choose a sequence of points in $\T(\Sigma')$
\[
z_{\Sigma'}=u_1, \ldots, u_N, 
\qquad
d_{\T(\Sigma')} \big( u_N, \Tt(\Sigma')\big) =\frac{\tau_z}{\eta_0}
\]
similar to claim 2 in \propref{Prop:Torus-Identity} and let $z^i \in \T(S)$ be a point 
so that
\[
z^i_\alpha = g(a)
\qquad\text{and}\qquad
z^i_{\Sigma} = u_i. 
\]
As before, we have (here $\tau_z$ plays the role of $L$ in claim 2 above)
\[
d_{\T(\alpha)} \Big( f(z^1)_\alpha, f(z^N)_\alpha \Big) 
   \leq \frac{2 \tau_z \Distance_1}{r}.
\]
Also, using the fact that points in the thick part move by at most $\Dt$
and $z^N$ has a distance $\frac{\tau_z}{\eta_0}$ to the thick part, by the triangle inequality, as in \propref{Prop:Torus-Identity},  we have
\[
d_{\T(\alpha)} \big( f(z^N)_{\alpha}, g(a) \big)  \leq 
d_\T\big( f(z^N), z^N) \leq (\mul+1) \frac{\tau_z}{\eta_0}  + \add + \Dt. 
\]
Noting that $f(z^1_\alpha) = \bg(a)$, by the triangle inequality, we have 
\begin{equation} \label{Eq:tau_z} 
d_{\T(\alpha)}\big(  \bg(a), g(a) \big) \leq 
   \left(\frac{2 \Distance_1}{r} + \frac{\mul + 1}{\eta_0} \right) \tau_z + (\add + \Dt). 
\end{equation} 
The same holds for $\bg(b)$ and $g(b)$. Note that $a < 0 < b$ and $|a|$ and $b$ 
are larger than $\tau_z/2$.  Assuming $\eta_0$ is large compared with $\mul$, and 
$\tau_z$ is large compared with the additive error in \eqnref{Eq:tau_z}, we have 
that the distance between the end points of $\bg$ and $g$ is much less than the 
length of $g$, and hence, $\bg(0)$ is uniformly close to $g$. 

Taking two such geodesics passing through $z_\alpha$ that make a definite 
angle with each other, we have $f(z)_\alpha$ is uniformly close to $z_\alpha$. 
\end{proof}

\begin{proof}[Proof of \thmref{Thm:Induction}] 
Let $z \in \T(\Sigma, L)$ be a point satisfying \eqnref{Eq:Far-Away}. 
Let $\Sigma'$ be the surface obtained from $\Sigma$ after removing
all subsurfaces of complexity one. Let $\alpha$ be the shortest curve 
in $z_{\Sigma'}$. Let $\Sigma'' = \Sigma' - \alpha$. 

By \propref{Prop:Torus-Identity}, for any component $W$ of $\Sigma$ with
$\xi(W) =1$ we have
\begin{equation}
\label{eq:W:small}
d_{\T(W)}\big( z_W, f(z)_W \big) \leq \Distance_W. 
\end{equation}
For $x' \in \T(\Sigma', L)$, the point $z'=(z_W, x')$ is point in $\T(\Sigma, L)$. 
Define for some $\bar L$, 
\[
f_{\Sigma'} \from \T(\Sigma', L) \to \T(\Sigma', \bar L) 
\]
by
\[
f_{\Sigma'} (x') = f(z')_{\Sigma'}.
\]
We will show that $f_{\Sigma'}$ is a quasi-isometry. It will also  turn out that 
$|\bar L-L |=O(1)$.  We first show that $f_{\Sigma'}$ is coarsely onto. This follows from 
\eqnref{eq:W:small} which says that under the map  $f$  the slice 
$\big\{(z_W,x'):x'\in \T(\Sigma',L)\big\}$ is mapped a bounded distance from itself. 
The same is true for $f^{-1}$ and the coarse onto statement follows. 

For $x', x'' \in \T(\Sigma', L)$, 
let $z' = (x', z_W)$ and $z'' = (x'', z_W)$. Then 
\begin{align*}
d_{\T(\Sigma')} \Big( f_{\Sigma'}(x'), f_{\Sigma'}(x'') \Big) 
  & = d_{\T(\Sigma')} \Big( f(z')_{\Sigma'}, f(z'')_{\Sigma'}\Big) \\
  & \leq d_{\T(\Sigma)} \Big( f\big(z'\big), f\big(z''\big) \Big) \\
  &\leq \mul \, d_{\T(\Sigma)} \big( z' ,z'' \big) + \add 
  = \mul  \, d_{\T(\Sigma')} \big( x',x'' \big) + \add.
\end{align*} 
Hence, we only need to check the lower bound. It is enough to prove this 
for points $x'$ and $x''$ that have a distance of at least $R$ from $\LR(\Sigma')$ and 
$\partial_L(\Sigma')$. Assuming $\Length_\Sigma$ is large enough, such points 
exists and form a connected subset of $\T(\Sigma', L)$. 
But, we have shown that 
\[
d_{T(W)} \big( f(z')_W , z'_W\big) \leq \Distance_W
\qquad\text{and}\qquad
d_{T(W)}( f(z'')_W , z''_W) \leq \Distance_W.
\]
Hence, 
\begin{align*}
d_{\T(\Sigma')} \Big( f_{\Sigma'}(x'), f_{\Sigma'}(x'') \Big) 
  & = d_{\T(\Sigma')} \Big( f(z')_{\Sigma'}, f(z'')_{\Sigma'} \Big) \\
 & \geq d_{\T} \big( f(z'), f(z'') \big) 
   - d_{\T(W)} \big( f(z')_{W}, f(z'')_{W} \big) \\
 & \geq \frac{1}{\mul}d_\T(z', z'') - \add - 2 \Distance_W \\
  & \ge \frac{1}{\mul}d_{\T(\Sigma')} (x', x'') - (\add + 2 \Distance_W). 
\end{align*}
That is, $f_{\Sigma'}$ is a $(\mul_{\Sigma'}, \add_{\Sigma'})$--quasi-isometry
for $\mul_{\Sigma'}= \mul$ and $\add_{\Sigma'} = \add + 2 \Distance_W$. 

Now, Propositions~\ref{Prop:Max-Preserved} and \ref{Prop:Ivanov} apply. In fact, 
since $f$ is anchored, we can conclude that $\fsP$ is the identity. 
Hence, \propref{Prop:Thick}, implies that $f_{\Sigma'}$ is in fact $\Dt$ close to identity 
in the thick part of $\T(\Sigma', L)$. Therefore, 
\begin{equation}
\label{eq:shortcurve}
d_{\T(\alpha)}\big( z_\alpha, f(z)_\alpha \big) \leq \Distance_{\rm Top}. 
\end{equation}

  Let $L'' = \tau_{z_{\Sigma'}} = \tau_z(\alpha)$.
Now, for any $u \in \T(\Sigma'', L'')$, let $z'_u \in \T(\Sigma', L'')$ be the point
that projects to $z_\alpha$ in $\T(\alpha)$ and projects 
to $u$ in $\Sigma''$. 
 Call the set of points $z_u'$ which project to $z_\alpha$ the   {\em slice} through $z_\alpha$.  At each point on the slice, $\alpha$ is the shortest curve.   For some $L'''$, we can define  
\[
f_{\Sigma''} \from \T(\Sigma'', L'') \to \T(\Sigma'', L''') 
\]
by
\[
f_{\Sigma''} (u) = f_{\Sigma'}(z_u')_{\Sigma''}.
\]
A similar argument to the one above shows that $f_{\Sigma''}$ is a quasi-isometry. Again 
by \eqnref{eq:shortcurve} both $f$ and $f^{-1}$ preserve the slice
up to bounded error which says the map $f_{\Sigma''}$ is coarsely onto.
The upper and lower bounds in the definition of quasi-isometry go as before. 

We next show that up to bounded additive error for any $z_u'$ in the slice, $\alpha$ 
is the shortest curve on $f(z_u')$. For let  $\beta$ the shortest curve.  Applying   
\eqnref{eq:shortcurve} to $f^{-1}$ we see that 
\[
\tau_{f(z_u')}(\alpha)\leq \tau_{f(z_u')}(\beta)
  \leq \tau_{z_u'}(\beta)+\Distance_{\rm Top}
  \leq \tau_{z_u'}(\alpha)+\Distance_{\rm Top}
  \leq \tau_{f(z_u')}(\alpha)+2\Distance_{\rm Top}.
\]
This says that $|L'''-L''|\leq  2\Distance_{\rm Top}$ and so by introducing a slightly larger 
additive error in the constants of quasi-isometry we can assume the image of our map is 
in $\T(\Sigma'',L'')$.

We now show that $f_{\Sigma''}$ is anchored. 
This is because if $u \in \partial_{L''}(\Sigma'')$, then for every $\beta \in P_u$, 
$\tau_u(\beta) = L''$ and every curve $\beta$ is the shortest curve in $z_u$. 
Then as above, the projection of $f_{\Sigma''}(u)=f_{\Sigma'}(z_u')$ to every $\beta$ 
is also close to the projection of  $z_u'$ to $\beta$ which in turn is the same as the 
projection of $u$ to $\beta$. That is $f_{\Sigma''}(u)$ is close to $u$. 

By induction, (\thmref{Thm:Induction} applied to $\Sigma''$), $f_{\Sigma''}$ is 
$\Distance_{\Sigma''}$--close to the identity. We have shows that the projections of $z$ 
to $\T(W)$, $\T(\alpha)$ and $\T(\Sigma'')$ are close to the projections of $f(z)$ to the 
same. That is, $f(z)$ is close to $z$. But the set of points satisfying \eqnref{Eq:Far-Away} 
is $R$--dense in $\T(\Sigma, L)$. Thus, for an appropriate value of $\Distance_\Sigma$, 
the theorem holds. 
\end{proof}

\begin{proof}[Proof of \thmref{Thm:Main}]
The proof is the same as above. From \propref{Prop:Ivanov} we have that
there is an isometry of $\T(S)$ so that if we precompose $f$ with this isometry, 
then $\fsP$ is the identity. Assuming this is done, \propref{Prop:Thick} implies that 
the restriction of $f$ to $\Tt$ is $\Dt$--close to the identity. 

Now consider a point $z \in \T(S)$ and let $\alpha$ be the shortest curve
in $z$. Choosing $R$ large enough so that statements in \secref{Sec:Short-Curves}
apply and \propref{Prop:Shortest-Identity} apply and so that $R \geq \Length_\Sigma$ 
for any subsurface $\Sigma$ of $S$. If $d_\T(z, \LR) \geq \mu \eta R$ then, 
applying \propref{Prop:Shortest-Identity} we have 
\begin{equation} \label{Eq:alpha-Main}
d_{\T(\alpha)} \big( z_\alpha, f(z)_\alpha \big) \leq \Distance_{\rm Top}. 
\end{equation} 
Let $\Sigma = S - \alpha$, let $L = \tau_z$ and as before, for some $\bar L$,  define a map 
\[
f_\Sigma \from \T(\Sigma, L) \to \T(\Sigma, \bar L)
\]
as follows: for $u \in \T(\Sigma, L)$ let $x \in \T(S)$ be a point so that,
\[
x_\Sigma = u 
\qquad\text{and}\qquad
x_\alpha = z_\alpha.
\]
Now, define
\[
f_\Sigma(u) = f(x)_\Sigma. 
\]
As we argued in the proof of \thmref{Thm:Induction}, this map is a quasi-isometry 
with uniform constants and it is anchored. Hence, by \thmref{Thm:Induction}, we have 
\begin{equation} \label{Eq:Sigma-Main}
d_{\T(\Sigma)}\big( z_\Sigma, f(z)_\Sigma \big) \leq D_\Sigma. 
\end{equation}
The theorem follows from Equations~\eqref{Eq:alpha-Main} and \eqref{Eq:Sigma-Main}. 
\end{proof}
w

\bibliographystyle{alpha}


\end{document}